\documentclass[11pt,leqno]{amsart}
\usepackage{amsmath,amssymb,amsthm}
\usepackage{hyperref}
\usepackage{amsfonts,amsmath, amsthm, amscd, amssymb, graphicx, color, mathrsfs, stmaryrd}
\usepackage[utf8]{inputenc}
\usepackage{tikz-cd}
\usepackage{adjustbox}

\usepackage{ulem} %%para tachado, eliminarlo luego

\usepackage{bbm}

\renewcommand{\geq}{\geqslant}
\renewcommand{\leq}{\leqslant}

\newcommand{\<}{\langle}
\renewcommand{\>}{\rangle}

\newcommand{\norm}[1]{\left\Vert#1\right\Vert}

\newcommand{\J}{\mathcal J}

\newcommand{\co}{\operatorname{co}}

\newcommand{\cconv}{\overline{\co}}
\newcommand{\dent}{\operatorname{dent}}
\newcommand{\ext}{\operatorname{ext}}

\newcommand{\strexp}{\operatorname{str-exp}}
\newcommand{\strext}{\operatorname{str-ext}}
\newcommand{\diam}{\operatorname{diam}}

\newtheorem{thm}{Theorem}[section]
\newtheorem{theo}[thm]{Theorem}
\newtheorem{prop}[thm]{Proposition}
\newtheorem{coro}[thm]{Corollary}
\newtheorem{lema}[thm]{Lemma}

\theoremstyle{definition}
\newtheorem{defi}[thm]{Definition}
\newtheorem{ejem}[thm]{Example}
\newtheorem{rema}[thm]{Remark}

\numberwithin{equation}{section}
\def\fnote#1{\footnote}

\def\ignora#1{}
\def\n3#1{\left\vert  \! \left\vert \! \left\vert \, #1 \, \right\vert \!
  \right\vert \! \right\vert }

\begin{document}

\title[]{ Extremal structure in ultrapowers of Banach spaces}

\author[L. Garc\'ia-Lirola]{Luis C. Garc\'ia-Lirola }
\address[L. García-Lirola]{Departamento de Matemáticas, Universidad de Zaragoza, 50009, Zaragoza, Spain} \email{\texttt{luiscarlos@unizar.es}}

\author[G. Grelier]{ Guillaume Grelier }
\address[G. Grelier]{Universidad de Murcia, Departamento de Matem\'aticas, Campus de Espinardo 30100 Murcia, Spain} \email{g.grelier@um.es}

\author [A. Rueda Zoca]{ Abraham Rueda Zoca }\address[A. Rueda Zoca]{Universidad de Murcia, Departamento de Matem\'aticas, Campus de Espinardo 30100 Murcia, Spain} \email{ abraham.rueda@um.es}
\urladdr{\url{https://arzenglish.wordpress.com}}



\maketitle

\begin{abstract}
Given a bounded convex subset $C$ of a Banach space $X$ and a free ultrafilter $\mathcal U$, we study which points $(x_i)_\mathcal U$ are extreme points of the ultrapower $C_\mathcal U$ in $X_\mathcal U$. In general, we obtain that when $\{x_i\}$ is made of extreme points (respectively denting points, strongly exposed points) and they satisfy some kind of uniformity, then $(x_i)_\mathcal U$ is an extreme point (respectively denting point, strongly exposed point) of $C_\mathcal U$. We also show that every extreme point of $C_{\mathcal U}$ is strongly extreme, and that every point exposed by a functional in $(X^*)_{\mathcal U}$ is strongly exposed, provided that $\mathcal U$ is a countably incomplete ultrafilter.
Finally, we analyse the extremal structure of $C_\mathcal U$ in the case that $C$ is a super weakly compact or uniformly convex set.
\end{abstract}

\section{Introduction}

The ultraproduct of Banach spaces has shown to be a very useful tool in the study of local properties of Banach spaces. For instance, in \cite[Theorem 11.1.4]{AlKa} ultrapowers are used in order to prove that a Banach space $X$ fails to have type $p>1$ if, and only if, $\ell_1$ is finitely representable in $X$. This link between the local structure of a Banach space $X$ and the global one of its ultrapowers $X_\mathcal U$ has allowed us to obtain structural results in Banach spaces (we refer the interested reader to \cite{heinrichsurvey}).

More recent studies about  the geometry of ultraproduct Banach spaces can be found in \cite{hardtke} for octahedral and almost square Banach spaces or in \cite{bksw, kw} for the Daugavet property. Actually, the example of the Daugavet property is paradigmatic of two basic facts that, more often than not, appear when dealing with a geometric property in Banach spaces. The first one is that, when requiring an ultrapower $X_\mathcal U$ to enjoy a geometric property, one has to look for a ``uniform version'' of this geometric property in $X$ (this happens for instance with the Daugavet property and the uniform Daugavet property \cite[Theorem~6.4]{bksw}, for the strict convexity and uniform convexity or for the reflexivity and superreflexivity \cite[Proposition~6.4]{heinrichsurvey}). The second one is that one should
avoid as much as possible to deal with the dual of an ultrapower space (this is done in \cite{bksw}  by using their Theorem 6.2). The reason is that, in most of the cases (i.e. out of superreflexive Banach spaces \cite[Corollary~7.2]{heinrichsurvey}), there is not good access to the dual of $X_\mathcal U$.

Taking the above two facts in mind, the aim of this paper is to study the extremal structure of subsets of an ultrapower. This structure codifies much information of bounded convex sets (we can think for instance in Krein-Milman theorems) and it is extremely useful in other areas of the Functional Analysis like the norm-attainment (see \cite{ccgmr,linds1963}). In the particular case of the extremal structure of ultrapowers, it has been previously considered by J.~Talponen in \cite{Talponen}, where the author studied the properties that link a point $x$ of the unit sphere of a Banach space and its image $\J(x)$ in the ultrapower through the canonical isometry. Some of his results will be generalised in this document since we deal with more general sets (not only with the unit ball) and more general ultrafilters (not only on $\mathbb N$). We also establish how the properties on the $x_i$'s are transferred to $(x_i)_\mathcal U$ and reciprocally. Then the results linking $x$ and $\J(x)$ are obtained as a particular case.

Let us now describe the content of the paper. In Section \ref{seccion:notacion} we include necessary terminology as well as a number of auxiliary results. These results are probably well known by specialists, but we include some proofs in order to be self-contained.

In Section \ref{seccion:principal} we provide our main results in complete generality. After providing a number of examples that suggests which properties we need to look for on $X$, we establish several stability results concerning the extremal structure. For instance, we extend Talponen's result in Theorem~\ref{th:jstrext} showing that $x$ is a strongly extreme point of a bounded convex set $C$ if and only if $\J(x)$ is a (strongly) extreme point of its ultrapower $C_{\mathcal U}$. Moreover, we show that extreme and strongly extreme points of $C_{\mathcal U}$ coincide under mild assumptions on $\mathcal U$, see Theorem~\ref{theo:extequistronglyCI}. We also characterise in Theorem~\ref{th:eqforallU} elements $(x_i)_{i\in I}\in C^I$ giving that $(x_i)_{\mathcal U}$ is an extreme point  of  $C_{\mathcal U}$ for every free ultrafilter $\mathcal U$ on $I$. In the context of denting points (respectively strongly exposed points) we prove that $(x_i)_\mathcal U\in C_\mathcal U$ is a denting point (respectively strongly exposed point) if $\{x_i\}$ satisfy a ``uniform denting condition'' (respectively a uniform condition of strong exposition), see Theorems~\ref{theo:condisufidenting} and~\ref{theo:carauniexposed}. Finally, we prove that every element of $C_{\mathcal U}$ which is exposed by a functional in $(X^*)_{\mathcal U}$ is in fact strongly exposed under mild assumptions on $\mathcal U$, see Theorem~\ref{theo:expequistronglyCI}.

In Section \ref{seccion:particulares} we take a closer look at
two particular kinds of weakly compact convex sets where we expect a nice behaviour of the extremal structure. On the one hand, we consider the notion of \textit{super weakly compact sets} $C$, those in which $C_\mathcal U$ is relatively weakly compact. The second, which is a subclass of super weakly compact sets, is the one of
the \textit{uniformly convex sets}. The main tool in this study is that if $C$ is uniformly convex then $C_\mathcal U$ too (see Proposition~\ref{unifconvex}). The aim of this section is to recover as much as possible the extremal properties of the unit ball of a uniformly convex Banach space. The biggest difficulty is that a uniformly convex set can have empty interior. However, we prove that any extreme point of such a set is denting and any exposed point is strongly exposed (Proposition~\ref{ext_unifconvex}). We also characterise the extreme points of its
ultraproduct set (Theorem~\ref{ext_ultrapower_unifconv}).

\section{Notation and auxiliary results}\label{seccion:notacion}

\subsection{Filters}\label{subsec:filters}

Given a set $I$ and $\mathcal F\subseteq \mathcal P(I)$, we say that $\mathcal F$ is a \textit{filter} if
\begin{enumerate}
    \item $\emptyset\notin\mathcal F$.
    \item $A,B\in\mathcal F$ implies $A\cap B\in \mathcal F$.
    \item $A\in \mathcal F$ and $A\subseteq B$ implies $B\in\mathcal F$.
\end{enumerate}
If $\mathcal U$ is a filter that is not contained in any other proper filter, we say that $\mathcal U$ is an \textit{ultrafilter}. 

Examples of ultrafilters are the \textit{principal ultrafilters}, which are those for which there exists $i\in I$ so that $A\in\mathcal U$ if, and only if, $i\in A$. We will be, however, interested in non-principal ultrafilters (or \textit{free ultrafilters}). These ultrafilters always exist by Zorn lemma, and they are characterised by the property of containing all the subsets of $I$ whose complement is finite.

We will use repeatedly the fact that every infinite subset $J\subset I$ belongs to a free ultrafilter over $I$. Indeed, consider $\mathcal B=\{J\cap L: I\setminus L \text{ is finite}\}$ and the collection $\mathcal F$ of subsets of $I$ containing an element of $\mathcal B$. It's easy to check that $\mathcal F$ is a free filter. By Zorn's lemma, $\mathcal F$  is contained in some ultrafilter (that is also free since it contains $\mathcal F$).

Given an ultrafilter $\mathcal U$ over $I$, a compact Hausdorff space $K$ and a function $f\colon I\longrightarrow K$, it is defined the \textit{limit of $f$ through $\mathcal U$}, and denoted by $\lim_\mathcal U f(i)$, as the unique $x\in K$ with the following property: for every neighbourhood $V$ of $x$,  the set $\{i\in I : f(i)\in V\}$ belongs to $\mathcal U$. In the case of $I=\mathbb N$, it can be proved that an ultrafilter $\mathcal U$ is free if, and only if, $\lim_\mathcal U x_n=\lim x_n$ for every convergent sequence $\{x_n\}$. The following result is a variation in the general context. Recall that $c_0(I)$ is the set of those bounded functions $f\colon I\longrightarrow \mathbb R$ for which $\{i\in I: \vert f(i)\vert\geq \varepsilon\}$ is finite for every $\varepsilon>0$.

\begin{lema}\label{lema:caraultralibreconvergencia}
Let $I$ be an infinite set. Let $f:I\to\mathbb R$ be a bounded function. The following assertions are equivalent:
\begin{enumerate}
    \item[(i)] $f\in c_0(I)$;
    \item[(ii)] $\lim_\mathcal Uf(i)=0$ for all free ultrafilter $\mathcal U$ on $I$.
\end{enumerate}
\end{lema}

\begin{proof}
Suppose that $(i)$ holds  and let $\varepsilon>0$. Let $\mathcal U$ be any free ultrafilter on $I$, thus it contains all the sets with finite complement. It follows that $\{i\in I :  |f(i)|\leq\varepsilon\}\in\mathcal U$, which proves that $\lim_\mathcal Uf(i)=0$.

Now suppose that $f\notin c_0(I)$. Then there exists $\varepsilon>0$ such  that $J:=\{i\in I:  |f(i)|>\varepsilon\}$ is infinite. Let $\mathcal U$ be a free ultrafilter on $I$ with $J\in \mathcal U$.
Since $J\in\mathcal U$, it is clear that $\lim_\mathcal Uf(i)\geq\varepsilon>0$.
\end{proof}

Since we are not restricting our attention to filters over $\mathbb N$, the following definition is important.

\begin{defi}
A free ultrafilter $\mathcal U$ on an arbitrary set $I$ is said to be countably incomplete (CI in short) if there exists a sequence $(I_n)_{n\in \mathbb N}\subset \mathcal U$ with $\bigcap_n I_n=\emptyset$.
\end{defi}
    
Note that every CI ultrafilter is free (otherwise, the intersection of elements in the ultrafilter would be non-empty). It is worth mentioning that, given an infinite set $I$, the fact that every free ultrafilter over $I$ is CI is a frequent phenomenon. For instance, it is known \cite[Theorem 2.5]{hj} that if $\kappa$ is a cardinal with a free ultrafilter which is not CI, then $\kappa$ is a strongly inaccessible cardinal (see \cite{hj} for background).

We set a (well-known) characterization of CI ultrafilters. We include a short proof for completeness.

\begin{prop}\label{CI}
Let $\mathcal U$ be ultrafilter on a set $I$. The following assertions are equivalent:
\begin{enumerate}
    \item[(i)] $\mathcal U$ is CI;
    \item[(ii)] there exists $(a_i)_{i\in I}\subset\mathbb R$ such that $\lim_\mathcal Ua_i=0$ and $a_i>0$ for all $i\in I$.
\end{enumerate}
\end{prop}

\begin{proof}
Suppose first that $\mathcal U$ is CI and let $(I_n)_{n\geq 1}\subset \mathcal U$ with $\bigcap_{n\geq 1} I_n=\emptyset$. We can suppose that $I_1=I$ and $I_{n+1}\subsetneq I_n$ for all $n\in\mathbb N$. Define $(a_i)_{i\in I}$ by $a_i=\frac{1}{n}$ if $i\in I_n\setminus I_{n+1}$. Given 
$\varepsilon>0$, take $n_0$ such that $\frac{1}{n_0}<\varepsilon$. It is easily seen that $$I_{n_0}\subset\{i\in I:  a_i<\varepsilon\}.$$ It follows that the last set belongs to $\mathcal U$. Thus, $\lim_\mathcal Ua_i=0$.

Now suppose that $(ii)$ holds. For $n>0$, define $$I_n:=\left\{i\in I:  a_i<\frac{1}{n}\right\}\in\mathcal U.$$ Since $a_i\neq 0$ for all $i\in I$, it is clear that $\bigcap_n I_n=\emptyset$, i.e. $\mathcal U$ is CI.
\end{proof}

We need to recall the notion of the product ultrafilter. If $\mathcal U$ is an ultrafilter on $I$ and $\mathcal V$ is an ultrafilter on $J$ then $\mathcal U\times \mathcal V$ is the ultrafilter on $I\times J$ defined by $$L\in\mathcal U\times\mathcal V\iff\left\{j\in J\ | \left\{i\in I\ |\ (i,j)\in L\right\}\in\mathcal U\right\}\in\mathcal V.$$ 

The following lemma will be useful.

\begin{lema}\label{limit}
Let $(X,d)$ be a metric space, $\mathcal U$ an ultrafilter on a set $I$ and $\mathcal V$ an ultrafilter on a set $J$. Let $(x_{i,j})_{i,j}\in X^{I\times J}$. Then $\lim_{\mathcal U\times\mathcal V}x_{i,j}=\lim_{\mathcal V,j}\lim_{\mathcal U,i}x_{i,j}$ whenever all of these limits exist.
\end{lema}

\begin{proof}
Let $x=\lim_{\mathcal U\times\mathcal V}x_{i,j}$ and $y=\lim_{\mathcal V,j}\lim_{\mathcal U,i}x_{i,j}$. For $j\in J$, define also $y_j=\lim_{\mathcal U,i}x_{i,j}$. Fix $\varepsilon>0$ and note that by definition of the limit we have $$\{(i,j)\in I\times J\ |\ d(x_{i,j},x)<\varepsilon\}\in\mathcal U\times\mathcal V,$$ that is $$J_\varepsilon:=\{j\in J\ |\ \{i\in I\ |\ d(x_{i,j},x)<\varepsilon\}\in\mathcal U\}\in\mathcal V.$$ Then, for all $j\in J_\varepsilon$, we have that $\{i\in I\ |\ d(x_{i,j},x)<\varepsilon\}\in\mathcal U$ which implies that $d(y_j,x)\leq\varepsilon$. Since $J_\varepsilon\in\mathcal V$, we obtain that $d(x,y)\leq\varepsilon$. The arbitrariness of $\varepsilon$ allows us to conclude that $x=y$.
\end{proof}

\subsection{Banach spaces and ultraproducts}\label{subsecnotultra}

Given a Banach space $X$, we denote by $B_X$ (respectively $S_X$) its closed unit ball (respectively unit sphere). Also, by $X^*$ we denote the topological dual of $X$. For simplicity we will deal with real Banach spaces.

Given a Banach space $X$ and an infinite set $I$, we denote  $\ell_\infty(I,X):=\{f\colon I\longrightarrow X: \sup_{i\in I}\Vert f(i)\Vert<\infty\}$. Given a free ultrafilter $\mathcal U$ over $I$, consider $N_\mathcal U:=\{f\in \ell_\infty(I,X): \lim_\mathcal U \Vert f(i)\Vert=0\}$. The \textit{ultrapower of $X$ with respect to $\mathcal U$} is
the Banach space
$$X_\mathcal U:=\ell_\infty(I,X)/N_\mathcal U.$$
We will naturally identify a bounded function $f\colon I\longrightarrow X$ with the element $(f(i))_{i\in I}$. In this way, we denote by $(x_i)_{\mathcal U,i}$ or simply by $(x_i)_\mathcal U$, if no confusion is possible, the coset in $X_\mathcal U$ given by $(x_i)_{i\in I}+N_\mathcal U$.

From the definition of the quotient norm, it is not difficult to prove that $\Vert (x_i)_\mathcal U\Vert=\lim_\mathcal U \Vert x_i\Vert$ for every $(x_i)_\mathcal U\in X_\mathcal U$. This implies that the canonical inclusion $j:X\longrightarrow X_\mathcal U$ given by the equation
$$\J(x):=(x)_\mathcal U$$
is an into linear isometry.

Note that, if $X$ is finite dimensional, then the previous mapping $j$ has an inverse $j^{-1}$ given by $j^{-1}(x_i):=\lim_\mathcal U x_i$. Consequently, $X_\mathcal U=X$ isometrically.

Note also that there is another inclusion $(X^*)_\mathcal U\longrightarrow (X_\mathcal U)^*$ given by the action
$$\langle (x_i^*)_\mathcal U, (x_i)_\mathcal U\rangle=\lim_\mathcal U x_i^*(x_i).$$
Since the previous map can be (easily) proved to be a linear isometry, we will consider $(X^*)_\mathcal U$ as a subspace of  $(X_\mathcal U)^*$. Recall that $(X^*)_\mathcal U=(X_\mathcal U)^*$ if, and only if, $X$ is superreflexive \cite[Corollary~7.2]{heinrichsurvey}. However, in general, $(X^*)_\mathcal U$ is norming for $X_{\mathcal U}$. In fact, given $(x_i)_\mathcal U\in X_\mathcal U$, if we pick for all $i\in I$ an element $x^*_i\in S_{X^*}$ such that $x^*_i(x_i)=\|x_i\|$, we get $$\langle(x^*_i)_\mathcal U,(x_i)_\mathcal U\rangle=\lim_\mathcal Ux^*_i(x_i)=\lim_\mathcal U\|x_i\|=\|(x_i)_\mathcal U\|.$$
Given a bounded subset $A\subseteq X$, denote $A_\mathcal U:=\{(x_i)_\mathcal U: x_i\in A\ \forall i\in I\}$. It is obvious that $A_\mathcal U$ is also bounded. Moreover $A$ is convex if and only if $A_\mathcal U$ is convex. Concerning the question of when $A_\mathcal U$ is closed we have the following result. The first part can be found in \cite[ Proposition~3.2]{SWC1}.

\begin{prop}\label{fact:closed} Let $X$ be a Banach space and let $A\subset X$ be a bounded set. Let $\mathcal U$ be a CI ultrafilter on a infinite set $I$. Then $A_{\mathcal U}$ is closed and $(\overline{A})_\mathcal U=A_\mathcal U=\overline{A_\mathcal U}$.
\end{prop}

\begin{proof}
Let $(I_n)_{n\in\mathbb N}$ be a sequence of sets as in the definition of CI ultrafilter, we may assume that $I_{n+1}\subsetneq I_n$ for every $n$. Let $x\in\overline{A_{\mathcal U}}$ and let $(x^n)_{n\in\mathbb N}$ be a sequence of $A_{\mathcal U}$ such that $\|x-x^n\|<\frac{1}{n}$. Consider $U_n=\{i\in I\ :\ \|x_i-x^n_i\|<\frac{1}{n}\}$ for all $n$ and note that $U_n\in \mathcal U$. Then define $I'_n=I_n\cap U_n\in \mathcal U$. Define $y\in A_{\mathcal U}$ by $y_i=x^m_i$ if $i\in I'_m\setminus I'_{m+1}$ for some $m$, and $y_i=x_0$ in the other case, where $x_0$ is a arbitrary element of $A$. One can check that $x=y\in A_{\mathcal U}$. 

Now we need to show that $(\overline{A})_\mathcal U=A_\mathcal U$. Let $(x_i)_\mathcal U\in (\overline{A})_\mathcal U$ with $x_i\in\overline{A}$ for all $i\in I$. By Proposition~\ref{CI}, there exists $(a_i)_{i\in I}\subset\mathbb R$ such that $\lim_\mathcal Ua_i=0$ and $a_i>0$ for all $i\in I$. For $i\in I$, choose $y_i\in A$ such that $\|x_i-y_i\|<a_i$. We deduce that $\lim_\mathcal U\|x_i-y_i\|=0$, that is, $(x_i)_\mathcal U=(y_i)_\mathcal U\in A_\mathcal U$. The other inclusion is obvious.
\end{proof}

We end the subsection with another technical result which will be useful in order to deal with slices in an ultrapower. 

\begin{lema}\label{sup}
Let $A$ be a bounded subset of a Banach space $X$ and $\mathcal U$ be a free ultrafilter on an infinite set $I$. If $(x_i^*)_\mathcal U\in (X^*)_\mathcal U$, then $\sup_{A_\mathcal U} (x_i^*)_\mathcal U=\lim_\mathcal U\sup_Ax_i^*$.
\end{lema}

\begin{proof}
Let $a=\sup_{A_\mathcal U} (x_i^*)_\mathcal U$ and $b=\lim_\mathcal U\sup_Ax_i^*$. Let $\varepsilon>0$. By definition of $a$, there exists $(x_i)_\mathcal U\in A_\mathcal U$ such that $\langle(x_i^*)_\mathcal U,(x_i)_\mathcal U\rangle>a-\varepsilon$. Then there exists $J\in\mathcal U$ such that $x^*_i(x_i)>a-\varepsilon$ for all $i\in J$. It follows that $$\sup_Ax_i^*\geq x^*_i(x_i)>a-\varepsilon,$$ for all $i\in J$ and taking limit on $\mathcal U$, we conclude that $b\geq a-\varepsilon$. 

Now, by definition of $b$, the set $J:=\{i\in I: \sup_A x_i^*>b-\varepsilon\}$ belongs to $\mathcal U$. For all $i\in J$ there exists $x_i\in A$ such that $x_i^*(x_i)>b-\varepsilon%,b+\varepsilon)
$. Define $y_i=x_i$ if $i\in J$ and $y_i=x_0$ if not, where $x_0$ is an arbitrary element of $A$. It is clear that $$\langle(x_i^*)_\mathcal U,(y_i)_\mathcal U\rangle\geq  b-\varepsilon%,b+\varepsilon]
.$$ It follows that $a\geq b-\varepsilon$.
\end{proof}

\subsection{Extremal structure in Banach spaces}
    
Let $C$ be a bounded convex subset of a Banach space $X$. The set of extreme points of $C$ is denoted by $\ext(C)$. Recall that a point $x\in C$ is strongly extreme if for all sequences $(y_n)_n,(z_n)_n\subset C$ such that $$\left\|x-\frac{y_n+z_n}{2}\right\|\xrightarrow[n]{}0,$$ one has that $\|y_n-z_n\|\xrightarrow[n]{}0$. The set of strongly extreme points of $C$ is denoted by $\strext(C)$. 

A slice of $C$ is a subset of $C$ defined by $$S(C,x^*,\alpha)=\{x\in C\ |\ x^*(x)>\sup_Cx^*-\alpha\}$$ where $x^*\in X^*$ and $\alpha>0$. 

Let $Z$ be a subspace of $X^*$. A point $x\in C$ is a $Z$-denting point if for all $\varepsilon>0$, there exist $x^*\in Z$ and $\alpha>0$ such that $x\in S(C,x^*,\alpha)$ and $\diam(S(C,x^*,\alpha))<\varepsilon$. We denote it by $x\in\dent_Z(C)$. A $X^*$-denting point is simply called denting and we write $\dent(C)=\dent_{X^*}(C)$.

A point $x\in C$ is a $Z$-exposed point if there exists $x^*\in Z$ such that $x^*(x)>x^*(y)$ for all $y\in C\setminus\{x\}$. We also said that $x^*$ exposes $x$ in $C$. The set of $Z$-exposed point of $C$ is denoted by $\exp_Z(C)$. A point $x\in C$ is said $Z$-strongly exposed if there exists $x^*\in Z$ exposing $x$ and such that for all sequences $(x_n)_n\subset C$ such that $x^*(x_n)\xrightarrow[n]{}x^*(x)$, it follows that $x_n\xrightarrow[n]{}x$. In this case, we write $x\in\strexp_Z(C)$. It is easy to show that $x\in C$ is $Z$-strongly exposed if there exists $x^*\in Z$ such that $x\in S(C,x^*,\alpha)$ for all $\alpha>0$ and $\lim_{\alpha\to 0^+}\diam(S(C,x^*,\alpha))=0$. As before, an $X^*$-(strongly) exposed point is said (strongly) exposed and we write $\exp(C)=\exp_{X^*}(C)$ and $\strexp(C)=\strexp_{X^*}(C)$. Obviously, \[\strexp(C)\subset \dent(C)\subset\strext(C)\subset\ext(C).\]

The following lemma will be applied with $Z=(X^*)_\mathcal U$ in Section~\ref{seccion:particulares}.

\begin{lema}\label{lemma:dentingsubnormante}
Let $C$ be a weakly compact convex subset of a Banach space $X$. Let $Z$ be a subspace of $X^*$ such that $(X,Z)$ is a dual pair.  Let $x$ be a point of a slice $S$ of $C$. Then there exists a slice $S'$ of $C$ defined by an element of $Z$ such that $x\in S'\subset S$. In particular, if $x\in C$ is denting then $x$ is $Z$-denting.
\end{lema} 

%\mnote{\lcch{No sé si tiene utilidad, pero parece que la hipótesis de convexidad en $C$ no es necesaria. Creo que se puede quitar usando Krein-Smulyan para $C\cap S^c$}}

\begin{proof}
Note that $C \setminus S$ is a weakly compact set which does not contain $x$. In particular, $C\setminus S$ is $\sigma(X,Z)$-compact. Since $Z$ separates points of $X$, the topology $\sigma(X,Z)$ is Hausdorff. By the Hahn-Banach theorem, there exists $x^*\in (X,\sigma(X,Z))^*=Z$ and a slice $S'$ of $C$ defined by $x^*$ such that $(C\setminus S)\cap S'=\emptyset$ and $x\in S'$. It follows that $x\in S'\subset S$.
\end{proof}

\section{Main results}\label{seccion:principal}

In this section we will exhibit general results about extremality in ultrapowers. Let us begin with some illustrative examples which reveal how restrictive the structure of $X_\mathcal U$ is for extremal notions. Let $C$ be a bounded, closed and convex subset of $X$. The following example shows that $(x_i)_\mathcal U$ needs not to be an extreme point of $C_{\mathcal U}$  
even if $x_i$ is strongly exposed in $C$ for every $i$.

\begin{ejem}\label{ejemplo:compactofexpuestolimnoextre} Let $X=\mathbb R^3$ and $$C:=\cconv(\{(x,y,z)\in\mathbb R^3: (x-1)^2+y^2=1, z=0\}\cup\{(0,0,1),(0,0,-1)\}).$$ Take a sequence $u_n\in\{(x,y,z)\in\mathbb R^3: (x-1)^2+y^2=1, z=0\}\setminus\{(0,0,0)\}$ that converges to $u=(0,0,0)$ . It is not difficult to prove that $u_n\in \strexp(C)$ for every $n\in\mathbb N$ and $u\notin\ext(C)$. Let $\mathcal U$ be a free ultrafilter over $\mathbb N$ and notice that $(u_n)_\mathcal U=\J(u)$ (see Subsection \ref{subsecnotultra}). Since $u$ is not an extreme, $j$ is an onto linear isometry and $C_\mathcal U=\J(C)$, we get that $\J(u)=(x_n)_\mathcal U$ is not an extreme point of $C_{\mathcal U}$.
\end{ejem}

In view of the previous example, we might wonder whether $(x_i)_\mathcal U$ is exposed in $C_{\mathcal U}$ if we require that $\lim_\mathcal U x_i$ is extreme in $C$. However, this is not the case.

\begin{ejem}
Consider in $X=\mathbb R^2$ the sets  $K_1:=\{(x,y)\in\mathbb R^2: x^2+y^2\leq 1, x\leq 0\}$,  $K_2:=\{(x,y)\in\mathbb R^2: (x-1)^2+y^2\leq 1, x\geq 1\}$, and $C:=\co(K_1\cup K_2)$, which is a compact convex set. Consider $(x_n,y_n):=\left(-\sqrt{\frac{2}{n}-\frac{1}{n^2}},1-\frac{1}{n}\right)$ for every $n\in\mathbb N$. It can be proved that every $(x_n,y_n)$ is strongly exposed in $C$ by $(x_n^*,y_n^*)=(x_n,y_n)$. However, the point $(0,1)$, which is limit of $\{(x_n,y_n)\}$, is not exposed in $C$. With a similar argument to that of Example~\ref{ejemplo:compactofexpuestolimnoextre} we conclude that, given any free ultrafilter $\mathcal U$ over $\mathbb N$, $(x_n,y_n)_\mathcal U$ is not a exposed point of $C_{\mathcal U}$.
\end{ejem}

As we pointed out in the introduction, a Banach space $X$ is uniformly convex if, and only if, every ultraproduct of $X$ is strictly convex. In that sense, one might expect that, if $X$ is strictly convex, then one can find an ultraproduct $X_\mathcal U$ so that the unit ball contains at least an extreme point. Our last example reveals that this is also false. 

\begin{ejem}
There is a strictly convex Banach space $X$ so that $B_{X_\mathcal U}$ does not have any extreme point for every infinite set $I$ and every free ultrafilter $\mathcal U$ over $I$. Indeed, let $X$ be a strictly convex Banach space which $M$-embedded in its bidual (see e.g. \cite[P. 168]{HWW1993}). In particular, $X$ is \textit{almost square} \cite[Corollary~4.3]{All2016}. However, \cite[Proposition~2.12]{hardtke} (it is established for $I=\mathbb N$, but it can be settled in general) shows that, given $I$ and $\mathcal U$, then for every $(x_i)_\mathcal U\in S_{X_\mathcal U}$ there exists $(y_i)_\mathcal U\in S_{X_\mathcal U}$ so that $\Vert (x_i)_\mathcal U \pm (y_i)_\mathcal U\Vert=1$. From here the absence of extreme points of $B_{X_\mathcal U}$ is clear.
\end{ejem}

The previous examples reveal that the presence of extremal structure in an ultrapower is very restrictive. Because of this reason, we will look for uniform conditions on the $x_i$'s.

\subsection{Extreme and strongly extreme points}

We will begin by exploring the extreme points of a set $C_\mathcal U$ for a given bounded closed and convex subset $C$ of $X$. Let us start with the following characterization of extreme points of $C_{\mathcal U}$.

\begin{theo}\label{prop:caraextrefiltrofijo}
Let $C$ be a bounded closed convex subset of a Banach space $X$ and $\mathcal U$ be a free ultrafilter on an infinite set $I$. Let $(x_i)_\mathcal U\in C_\mathcal U$. The following assertions are equivalent:
\begin{enumerate}
    \item[(i)] $(x_i)_\mathcal U\in\ext(C_\mathcal U)$;
    \item[(ii)] for any $(y_i)_\mathcal U,(z_i)_\mathcal U\in C_\mathcal U$ so that $\lim_\mathcal U\Vert x_i-\frac{y_i+z_i}{2}\Vert=0$, it follows that $\lim_\mathcal U \Vert x_i-y_i\Vert=0$ and $\lim_\mathcal U\Vert x_i-z_i\Vert=0$.
\end{enumerate}
\end{theo}

\begin{proof}
Given $(y_i)_\mathcal U,(z_i)_\mathcal U\in C_\mathcal U$, notice that $\lim_\mathcal U\Vert x_i-\frac{y_i+z_i}{2}\Vert=0$ if, and only if,  $(x_i)_\mathcal U=\frac{1}{2}( (y_i)_\mathcal U+(z_i)_\mathcal U)$ in $C_\mathcal U$. On the other hand,  
\[(x_i)_\mathcal U=(y_i)_\mathcal U\Leftrightarrow (x_i-y_i)_\mathcal U=(0)_\mathcal U\Leftrightarrow \lim_\mathcal U\Vert x_i-y_i\Vert=0.\]
This gives the characterisation.
\end{proof}

Given $(x_i)_{i\in I}\in C^I$, it is not difficult to realise that being a (strongly) extreme point of $C_{\mathcal U}$ depends on the considered ultrafilter $\mathcal U$ on $I$. For instance, just take $C=[-1,1]^2\subset \mathbb R^2$ and $x_n=(1,0)$ if $n$ is odd and $x_n=(1,1)$ if $n$ is even. It is easy to find  free ultrafilters $\mathcal U$ and $\mathcal V$ on $\mathbb N$ such that $(x_n)_{\mathcal U}\in \ext(C_{\mathcal U})$ but $(x_n)_{\mathcal V}\notin \ext(C_{\mathcal V})$. Our next goal is to characterise when $(x_i)_{\mathcal U}$ is a strongly extreme point \textit{for every} free ultrafilter $\mathcal U$ in terms of the space $c_0(I)$. Note that this result will be improved in Theorem~\ref{th:eqforallU} below.

\begin{theo}\label{theo:caraextretodoultra}
Let $C$ be a bounded convex subset of a Banach space $X$. Let $I$ be a infinite set and $(x_i)_{i\in I}\in C^I$. The following assertions are equivalent:
\begin{enumerate}
    \item[(i)] $(x_i)_\mathcal U\in\ext(C_\mathcal U)$ for every free ultrafilter $\mathcal U$ on $I$;
    \item[(ii)]  for any $(y_i)_{i\in I},(z_i)_{i\in I}\in C^I$ so that $\left(\left\|x_i-\frac{y_i+z_i}{2}\right\|\right)_{i\in I}\in c_0(I)$, it follows that $(\|x_i-y_i\|)_{i\in I}\in c_0(I)$ and $(\Vert x_i-z_i\Vert)_{i\in I}\in c_0(I)$.
\end{enumerate}
\end{theo}

\begin{proof} 
For (i)$\Rightarrow$(ii) assume that (ii) does not hold. Then there exist $\varepsilon_0>0$ and elements $(y_i)_{i\in I},(z_i)_{i\in I}\in C^I$ so that $(\|x_i-\frac{y_i+z_i}{2}\|)_{i\in I}\in c_0(I)$ and the set $J:=\{i\in I: \Vert x_i-y_i\Vert\geq \varepsilon_0\}$ is infinite. Let $\mathcal U$ be a free ultrafilter on $I$ with $J\in\mathcal U$.  Note that $\lim_\mathcal U \Vert x_i-\frac{y_i+z_i}{2}\Vert=0$ by Lemma~\ref{lema:caraultralibreconvergencia}, that is, $(x_i)_{\mathcal U}=\frac{(y_i)_{\mathcal U}+(z_i)_{\mathcal U}}{2}$. Moreover, it is not difficult to prove that $\lim_\mathcal U \Vert x_i-y_i\Vert\geq \varepsilon_0$, so by Theorem~\ref{prop:caraextrefiltrofijo} we have that $(x_i)_\mathcal U$ is not an extreme point of $C_\mathcal U$.

For (ii)$\implies$(i) assume that (i) does not hold, that is, there exists a free ultrafilter $\mathcal U$ over $I$ so that $(x_i)_\mathcal U$ is not an extreme point of $C_\mathcal U$. By Theorem~\ref{prop:caraextrefiltrofijo} there are $(y_i)_{i\in I},(z_i)_{i\in I}\in C^I$ so that $\lim_\mathcal U \Vert x_i-\frac{y_i+z_i}{2}\Vert=0$ but $\lim_\mathcal U \Vert x_i-y_i\Vert> \varepsilon_0$ for certain $\varepsilon_0>0$. This implies that the set $B:=\{i\in I: \Vert x_i-y_i\Vert>\varepsilon_0\}$ belongs to $\mathcal U$. Now, construct inductively a sequence $$i_n\in \left\{i\in I: \left\| x_i-\frac{y_i+z_i}{2}\right\|<\frac{1}{n}\right\}\cap B \setminus\{i_1,\ldots, i_{n-1}\}$$
(the intersection is non-emtpy because the previous set actually belongs to $\mathcal U$). Define $(y'_i)_{i\in I}$ and $(z'_i)_{i\in I}$ by $y_i'=z_i'=x_i$ if $i\notin \{i_n: n\in\mathbb N\}$ and $y_{i_n}'=y_{i_n}$ and $z_{i_n}'=z_{i_n}$ for every $n\in\mathbb N$. We have that $(\Vert x_i-\frac{y_i'+z_i'}{2}\Vert)_{i\in I}\in c_0(I)$ because clearly
$$\left\{i\in I: \left\| x_i-\frac{y_i'+z_i'}{2}\right\|\geq \frac{1}{j}\right\}\subseteq \{i_1,\ldots, i_{j-1}\} \quad \forall j\in \mathbb N.$$
On the other hand, we clearly have $\{i_n: n\in\mathbb N\}\subseteq \{i\in I: \Vert x'_i-y'_i\Vert\geq \varepsilon_0\}$, which implies that $(\Vert x_i'-y_i'\Vert)_{i\in I}\notin c_0(I)$, and then $(ii)$ does not hold. This completes the proof of $(ii)\implies(i)$. 
\end{proof}

Our aim is now to study when $(x_i)_\mathcal U$ is a strongly extreme point of $C_\mathcal U$. Let us start with the following result.

\begin{prop}\label{prop:extsubseq}
Let $C$ be a bounded convex set of a Banach space $X$ and $\mathcal U$ be a free ultrafilter on an infinite set $I$. Let $(x_i)_{i\in I}\in C^I$. 
\begin{enumerate}
    \item[(a)] Assume that for every countable subset $J=\{i_n:n\in \mathbb N\}\subset I$ there is a free ultrafilter $\mathcal V$ on $J$ such that $(x_{i_n})_{\mathcal V}\in \ext(C_{\mathcal V})$. Then $(x_i)_{\mathcal U}\in \strext(C_{\mathcal U})$.
    \item[(b)] Assume that $I=\mathbb N$ and $(x_n)_n\subset C$ is such that for all subsequences $(x'_n)_n$ of $(x_n)_n$ there is a free ultrafilter $\mathcal V$ on $\mathbb N$ such that $(x'_n)_{\mathcal V}\in \ext(C_{\mathcal V})$. Then $(x_n)_\mathcal U\in\strext(C_\mathcal U)$.
\end{enumerate}
\end{prop}

\begin{proof}
$(a)$ Suppose that $(x_i)_\mathcal U\notin\strext(C_\mathcal U)$. Then there exist two sequences $((y_i^n)_{\mathcal U,i})_n,((z_i^n)_{\mathcal U,i})_n\subset C_\mathcal U$ such that $$\frac{(y_i^n)_{\mathcal U,i}+(z_i^n)_{\mathcal U,i}}{2}\xrightarrow[n]{} (x_i)_\mathcal U$$ but $\|(y_i^n)_{\mathcal U,n}-(z_i^n)_{\mathcal U,i}\|>\varepsilon$ for all $n\in\mathbb N$ and for some $\varepsilon>0$. Without loss of generality, we can suppose that $\|y_i^n-z_i^n\|>\varepsilon$ for all $i\in I$ and $n\in\mathbb N$.

This means that for every $k\in \mathbb N$ there exists $n_k$ such that for all $n\geq n_k$ we have 
\[ \norm{ (x_i)_{\mathcal U}- \frac{(y_i^n)_{\mathcal U}+(z_i^n)_{\mathcal U}}{2}}<\frac{1}{k},\]
that is, the set
\[ A_{n,k}:=\left\{i\in I: \norm{ x_i-\frac{y_i^n+z_i^n}{2}}<\frac{1}{k}\right\}\]
belongs to $\mathcal U$ (in particular, it is infinite) for every $k$ and every $n\geq n_k$. 

Thus, we may choose $i_k\in A_{n_k, k}\setminus\{i_1, \ldots, i_{k-1}\}$ inductively. Then $J=\{i_k:k\in \mathbb N\}$ is countable and \[ \norm{ x_{i_k}-\frac{y_{i_k}^{n_k}+z_{i_k}^{n_k}}{2}} < \frac{1}{k} \qquad \forall k\in \mathbb N.\]

This implies that $\left(\norm{ x_{i_k}-\frac{y_{i_k}^{n_k}+z_{i_k}^{n_k}}{2}}\right)_{n\in \mathbb N}\in c_0(\mathbb N)$. So, given any free ultrafilter $\mathcal V$ on $\mathbb N$, we have  
\[ (x_{i_k})_{\mathcal V}= \frac{(y_{i_k}^{n_k})_{\mathcal V}+(z_{i_k}^{n_k})_{\mathcal V}}{2}\] by Lemma~\ref{lema:caraultralibreconvergencia}. Since $\norm{y_{i_k}^{n_k}-z_{i_k}^{n_k}}\geq \varepsilon$ for all $k$, it follows that $(x_{i_k})_{\mathcal V}\notin \ext(C_{\mathcal V})$, which is a  contradiction.

$(b)$ Just mimic the proof of $(a)$ noting that the sequence $(i_k)_k$ can be chosen to be strictly increasing.
\end{proof}

In \cite[Theorem~2.1]{Talponen} it was proved that, given $x\in S_X$ and a free ultrafilter $\mathcal U$ on $\mathbb N$, $\J(x)$ is an extreme point of $B_{X_\mathcal U}$ if, and only if, $x$ is strongly extreme. We can now improve  that result. 

\begin{thm}\label{th:jstrext}
Let $C$ be a bounded convex subset of a Banach space $X$, $\mathcal U$ be a CI ultrafilter on an infinite set $I$ and let $x\in C$. The following assertions are equivalent:
\begin{enumerate}
    \item[(i)] $\J(x)\in\strext(C_\mathcal U)$;  
    \item[(ii)] $\J(x)\in\ext(C_\mathcal U)$;
    \item[(iii)] $x\in\strext(C)$.
\end{enumerate}
\end{thm}

\begin{proof}
Obviously, we have that $(i)\implies (ii)$. Assume now that $x\in \strext( C)$. Given a countable subset $J=\{i_n:n\in \mathbb N\}\subset I$ and a free ultrafilter $\mathcal V$ on $\mathbb N$, Theorem~2.1 in \cite{Talponen} shows that $j_{\mathcal V}(x)\in \ext(C_{\mathcal V})$ (here $j_{\mathcal V}$ denotes the canonical embedding of $C$ into $C_{\mathcal V}$). Thus, the hypotheses of Proposition~\ref{prop:extsubseq} are satisfied, so we get that $\J(x)\in \strext(C_{\mathcal U})$. That is, (iii)$\Rightarrow$(i).

Finally, assume that $\J(x)\in\ext(C_{\mathcal U})$. If $x\notin\strext(C)$, then there sequences $(y_n)_n$, $(z_n)_n\subset C$ and a number $\varepsilon>0$ such that 
\[\lim_{n\to\infty} \norm{x-\frac{y_n+z_n}{2}}=0\]
and $\norm{x-y_n}\geq \varepsilon$ for all $n\in \mathbb N$. By extracting a subsequence, we may assume that $\norm{x-\frac{y_n+z_n}{2}}<\frac{1}{n}$ for all $n\in \mathbb N$. Now, let $(I_n)_n\subset \mathcal U$ be a sequence of sets with $I_1=I$,  $I_{n+1}\subsetneq I_n$ for all $n$ and $\bigcap_{n\in\mathbb N} I_n=\emptyset$, and define $(y_i')_{i\in I}$, $(z_i')_{i\in I}$ by $y_i'= y_n$ and $z_i'= z_n$ if $i\in I_n\setminus I_{n+1}$. Note that 
\[ I_n\subset \left\{i\in I: \norm{x-\frac{y_i'+z_i'}{2}}\leq \frac{1}{n}\right\}\qquad \forall n\in\mathbb N.\]
Since $I_n\in \mathcal U$, we get 
\[ \J(x)=\frac{(y_i')_{\mathcal U}+(z_i')_{\mathcal U}}{2}.\]
As $\J(x)$ is extreme, this implies $\J(x)=(y_i')_{\mathcal U}$. Since $\norm{x-y_i'}\geq\varepsilon$ for all $i\in I$, this is a contradiction. Therefore, $x\in \strext(C)$, as desired. This shows that (ii)$\Rightarrow$(iii). 
\end{proof}

Our aim will be to determine when $(x_i)_\mathcal U$ is a strongly extreme point of $C_\mathcal U$ which, thanks to the previous theorem, is equivalent to the fact that $j_{\mathcal V}((x_i)_{\mathcal U})$ is an extreme point of $(C_\mathcal U)_\mathcal V$ in the space $(X_\mathcal U)_\mathcal V$. Note that $(X_{\mathcal U})_{\mathcal V}$ is isometric to $X_{\mathcal U\times \mathcal V}$ (where $\mathcal U\times \mathcal V$ is the product ultrafilter defined in Subsection \ref{subsec:filters}). The isometry is $T\colon (X_\mathcal U)_\mathcal V\to X_{\mathcal U\times\mathcal V}$ defined by $T((x_{i,j})_{\mathcal U,i})_{\mathcal V,j})=(x_{i,j})_{\mathcal U\times\mathcal V}$ (see \cite[Proposition 2.1]{stern}). Moreover, it is clear that $T((A_\mathcal U)_\mathcal V)=A_{\mathcal U\times\mathcal V}$ for all bounded sets $A\subset X$.

\begin{thm}\label{theo:extequistronglyCI}
Let $C$ be a bounded convex subset of a Banach space $X$, $\mathcal U$ be a CI ultrafilter on an infinite set $I$. Then $\ext(C_\mathcal U)=\strext(C_\mathcal U)$.
\end{thm}

\begin{proof}
Let $(U_n)_n\subset\mathcal U$ be a strictly decreasing sequence of sets such that $\bigcap_{n>0}U_n=\emptyset$. Let $(x_i)_\mathcal U\in\ext(C_\mathcal U)$. We need to show that $(x_i)_\mathcal U\in\strext(C_\mathcal U)$. By Theorem \ref{th:jstrext}, it is enough to prove that $\J((x_i)_\mathcal U)\in\ext(C_{\mathcal U\times\mathcal U})$. Suppose that it is not true, then there exist $(y_{i,j})_{\mathcal U\times\mathcal U},(z_{i,j})_{\mathcal U\times\mathcal U}$ and $\varepsilon_0>0$ such that $\|(y_{i,j})-(z_{i,j})\|=\lim_{\mathcal U\times\mathcal U}\|y_{i,j}-z_{i,j} \|>\varepsilon_0$ and
$$\J((x_i)_\mathcal U)=\frac{(y_{i,j})_{\mathcal U\times\mathcal U}+(z_{i,j})_{\mathcal U\times\mathcal U}}{2}.$$ 
Up to changing the definition of $(y_{i,j})$ and $(z_{i,j})$ out of the set $\{(i,j)\in I^2\ |\ \| y_{i,j}-z_{i,j}\|>\varepsilon_0 \}\in\mathcal U\times \mathcal U$, we can assume that $\| y_{i,j}-z_{i,j}\|>\varepsilon_0$ holds for every $i,j\in I$.

It follows that $$\left\{(i,j)\in I^2\ |\ \left\|x_i-\frac{y_{i,j}+z_{i,j}}{2}\right\|<\frac{1}{n}\right\}\in\mathcal U\times\mathcal U$$ for all $n>0$, that is $$J_n:=\left\{j\in I\ | \left\{i\in I\ |\ \left\|x_i-\frac{y_{i,j}+z_{i,j}}{2}\right\|<\frac{1}{n}\right\}\in\mathcal U\right\}\in\mathcal U.$$ For $j\in I$ and $n>0$, define $I_{n,j}=\left\{i\in I\ |\ \left\|x_i-\frac{y_{i,j}+z_{i,j}}{2}\right\|<\frac{1}{n}\right\}$ and note that $I_{n,j}\in\mathcal U$ if $j\in J_n$. Since $J_1\in\mathcal U$, we have that $J_1\neq\emptyset$ and then let $j_1\in J_1$. Define $I'_1=I_{1,j_1}\cap U_1\in\mathcal U$. Now choose $j_2\in J_2$ and define $I'_2=I_{2,j_2}\cap I'_1\cap U_2\in\mathcal U$. Following by induction, we define $I'_n=I_{n,j_n}\cap I'_{n-1}\cap U_n\in\mathcal U$ where $j_n\in J_n$. For $i\in I$, define $y_i=y_{i,j_n}$ if $i\in I'_n\setminus I'_{n+1}$ for some $n>0$ and $y_i=x$ otherwise where $x$ is an arbitrary point of $C$. We define $z_i$ in the same way. Note that $\|y_i-z_i\|>\varepsilon_0$ for all $i\in I'_1$. In fact, if $i\in I'_1$, then there exists $n\geq 0$ such that $i\in I'_n\setminus I'_{n+1}$ (since $\bigcap_{n>0} I'_n=\emptyset$) and $\|y_i-z_i\|=\|y_{i,j_n}-z_{i,j_n}\|>\varepsilon_0$. We deduce that $\|(y_i)_\mathcal U-(z_i)_\mathcal U\|\geq\varepsilon_0$. Let's show that $(x_i)_\mathcal U=\frac{(y_i)_\mathcal U+(z_i)_\mathcal U}{2}$, which will contradict the extremality of $(x_i)_\mathcal U$ and will conclude the proof. Let $\varepsilon>0$ and take $n_0$ such that $\frac{1}{n_0}<\varepsilon$. We are going to show that $I'_{n_0}\subset\left\{i\in I\ |\ \left\|x_i-\frac{y_i+z_i}{2}\right\|<\varepsilon\right\}$, which implies that the last set belongs to $\mathcal U$. So let $i\in I'_{n_0}$. There exists $n\geq n_0$ such that $i\in I'_n\setminus I'_{n+1}$. Then we have that $$\left\|x_i-\frac{y_i+z_i}{2}\right\|=\left\|x_i-\frac{y_{i,j_n}+z_{i,j_n}}{2}\right\|<\frac{1}{n}\leq\frac{1}{n_0}<\varepsilon$$ and the proof is complete.
\end{proof}

Now, we are able to extend Theorem~\ref{theo:caraextretodoultra} including strongly extreme points. 

\begin{thm}\label{th:eqforallU}
Let $C$ be a bounded convex set of a Banach space $X$ and $I$ be an infinite set. Let $(x_i)_{i\in I}\in C^I$. The following assertions are equivalent:
\begin{enumerate}
    \item[(i)] $(x_i)_\mathcal U\in\strext(C_\mathcal U)$ for every free ultrafilter $\mathcal U$ on $I$;
    \item[(ii)] $(x_i)_\mathcal U\in\ext(C_\mathcal U)$ for every free ultrafilter $\mathcal U$ on $I$;
    \item [(iii)] for any $(y_i)_{i\in I},(z_i)_{i\in I}\in C^I$ so that $\left(\left\|x_i-\frac{y_i+z_i}{2}\right\|\right)_{i\in I}\in c_0(I)$, it follows that $(\|x_i-y_i\|)_{i\in I}\in c_0(I)$ and $(\Vert x_i-z_i\Vert)_{i\in I}\in c_0(I)$;
    \item[(iv)] $(x_j)_\mathcal V\in\ext(C_\mathcal V)$ holds for every countable subset $J\subseteq I$ and every free ultrafilter $\mathcal V$ over $J$.
\end{enumerate}
\end{thm}

\begin{proof}
(iv)$\implies$ (i) follows from (a) of Proposition~\ref{prop:extsubseq}. Moreover, (i)$\implies$(ii) is obvious, whereas (ii)$\Leftrightarrow$(iii) is Theorem \ref{theo:caraextretodoultra}.

Finally, let us prove (iii)$\Rightarrow$(iv), for which  we take a countable subset $J\subseteq I$ and, in order to prove (iv), by Theorem \ref{theo:caraextretodoultra}, take $(y_j),(z_j)\in C^J$ so that $(\Vert x_j-\frac{y_j+z_j}{2}\Vert)\in c_0(J)$, and let us prove that $(\Vert x_j-y_j\Vert)\in c_0(J)$. Define $y_i:=x_i$ and $z_i:=x_i$ if $i\in I\setminus J$, and it is obvious that $(\Vert x_i-\frac{y_i+z_i}{2}\Vert)\in c_0(I)$ since $(\Vert x_j-\frac{y_j+z_j}{2}\Vert)\in c_0(J)$.  Using (iii) we get that $(\Vert x_i-y_i\Vert)\in c_0(I)$ (analogously $(\Vert x_i-z_i\Vert)\in c_0(I)$). From here it is obvious that $(\Vert x_j-y_j\Vert)\in c_0(J)$ as desired. 
\end{proof}

\begin{rema}
Note that we have proved in Theorem \ref{theo:extequistronglyCI} that $\ext(C_\mathcal U)=\strext(C_\mathcal U)$ holds when $\mathcal U$ is a CI ultrafilter. The previous theorem shows that the hypothesis of countably incompleteness can be removed if we require $(x_i)_\mathcal V$ being extreme for every free ultrafilter $\mathcal V$.
\end{rema}

If $I=\mathbb N$, the previous theorem can be stated in terms of convergent sequences:

\begin{coro}
Let $C$ be a bounded convex subset of a Banach space $X$. Let $(x_n)_n\subset C$. The following assertions are equivalent:
\begin{enumerate}
    \item[(i)] $(x_n)_\mathcal U\in\strext(C_\mathcal U)$ for every free ultrafilter $\mathcal U$ on $\mathbb N$;
    \item [(ii)] $(x_n)_\mathcal U\in\ext(C_\mathcal U)$ for every free ultrafilter $\mathcal U$ on $\mathbb N$;
    \item[(iii)] For every pair of sequences $(y_n)_n, (z_n)_n$ in $C$, if $\|x_n-\frac{y_n+z_n}{2} \|\rightarrow 0$ then $\Vert x_n-y_n\Vert\rightarrow 0$ and $\Vert x_n-z_n\Vert\rightarrow 0$;
    \item[(iv)] $(x'_n)_n\in\ext(C_\mathcal U)$ for every subsequence $(x'_n)_n$ of $(x_n)_n$ and every free ultrafilter $\mathcal U$ on $\mathbb N$.
    \end{enumerate}
\end{coro}

\begin{proof}
$(i)\iff(ii)\iff(iii)\implies(iv)$ follows directly from the previous theorem. $(iv)\implies(i)$ is easily deduced from (b) of Proposition \ref{prop:extsubseq}.
\end{proof}

\subsection{Denting points}

In this subsection we will study the denting points in ultrapowers. To this end, let us consider the following notion.

\begin{defi}
Let $C$ be a bounded convex subset of a Banach space $X$. A subset $\{x_i\}_{i\in I}\subset C$ is said to be a \textit{uniformly denting set} if for every $\varepsilon>0$ there exists $\alpha_\varepsilon>0$ with the following property: for every $i\in I$ there exists $x_i^*\in S_{X^*}$ so that
$$\begin{array}{ccc}
x_i\in S(C,x_i^*,\alpha_\varepsilon) & \mbox{ and } & \diam(S(C,x_i^*,\alpha_\varepsilon))<\varepsilon.
\end{array}$$
\end{defi}

This definition should be compared with that of \cite{Tingley} and \cite[P. 4]{Talponen} of a uniform notion of dentable set.

Now we have the following result.

\begin{theo}\label{theo:condisufidenting}
Let $C$ be a bounded convex subset of a Banach space $X$, $\mathcal U$ be a free ultrafilter on an infinite set $I$ and $\{x_i\}_{i\in I}$ be a uniformly denting set in $C$. Then $(x_i)_\mathcal U\in\dent_{(X^*)_{\mathcal U}}(C_\mathcal U)$.
\end{theo}

For the proof we need the following lemma.

\begin{lema}\label{lemma:LCsuperl}
Let $C$ be a bounded convex subset of a Banach space $X$. Let $x^*\in S_{X^*}$ and $\alpha>0$. Then
$\diam(S(C,x^*,\frac{3}{2}\alpha))\leq 2 \diam(S(C,x^*,\alpha))$.
\end{lema}
\begin{proof}
Let $y,z\in S(C,x^*,\frac{3}{2}\alpha)$, and let us estimate $\Vert y-z\Vert$. To this end, pick $x\in S(C,x^*,\alpha/2)$ and notice that, if we call $\lambda:=\sup_Cx^*$, we get
$$x^*\left(\frac{x+y}{2}\right)>\frac{\lambda-\frac{\alpha}{2}+\lambda-\frac{3\alpha}{2}}{2}=\lambda-\alpha,$$
so $\frac{x+y}{2}\in S(C,x^*,\alpha)$. Similarly $\frac{x+z}{2}\in S(C,x^*,\alpha)$. Consequently
$$\diam(S(C,x^*,\alpha))\geq \left\Vert\frac{x+z}{2}-\frac{x+y}{2} \right\Vert=\frac{\Vert y-z\Vert}{2},$$
from where the result follows by the arbitrariness of $y,z$.
\end{proof}

\begin{proof}[Proof of Theorem~\ref{theo:condisufidenting}] Pick $\varepsilon>0$, and let us find a slice $S$ of $C_{\mathcal U}$ containing $(x_i)_\mathcal U$ whose diameter is smaller than or equal to $2\varepsilon$. To this end, since $\{x_i\}_{i\in I}$ is uniformly dentable we can find $\alpha>0$ and $\{x_i^*\}_{i\in I}\subset S_{X^*}$ so that $x_i\in S(C,x_i^*,\alpha)$ and $\diam(S(C,x_i^*,\alpha))<\varepsilon$. Note that
$$\langle(x_i^*)_\mathcal U,(x_i)_\mathcal U\rangle=\lim_\mathcal U x_i^*(x_i)\geq \lim_{\mathcal U} \sup_{C} x_i^*-\alpha,$$
so $(x_i)_{\mathcal U}\in S=S(C_\mathcal U, (x_i^*)_\mathcal U,\frac{3\alpha}{2})$. Now, in view of Lemma~\ref{lemma:LCsuperl}, it is enough to prove that $\diam(S(C_\mathcal U, (x_i^*)_\mathcal U,\alpha))\leq \varepsilon$. In order to do so, pick $(y_i)_\mathcal U,(z_i)_\mathcal U\in S(C_\mathcal U, (x_i)_\mathcal U,\alpha)$. Since $\langle(x_i^*)_\mathcal U,(y_i)_\mathcal U\rangle>\sup_{C_{\mathcal U}}(x_i^*)_{\mathcal U}-\alpha$ and $\langle(x_i^*)_\mathcal U,(z_i)_\mathcal U\rangle>\sup_{C_{\mathcal U}(x_i^*)_{\mathcal U}}-\alpha$, we get that
\[
    \lim_{\mathcal U} (x_i^*(y_i)-\sup_{C} x_i^*) + \alpha>0, \ \text{and} \quad   \lim_{\mathcal U} (x_i^*(z_i)-\sup_{C} x_i^*) + \alpha>0.
\]
Thus,
\[J:=\{i\in I: \min\{x_i^*(y_i),x_i^*(z_i)\}>\sup_{C}x_i^*-\alpha\}\in \mathcal U.\]
Given $i\in J$, we get that $y_i,z_i\in S(C,x_i^*,\alpha)$, and so $\Vert y_i-z_i\Vert<\varepsilon$. Consequently,
$$J\subseteq L:=\{i\in I: \Vert y_i-z_i\Vert<\varepsilon\},$$
so $L\in \mathcal U$. It is immediate to obtain that $\Vert (y_i)_\mathcal U-(z_i)_\mathcal U\Vert=\lim_\mathcal U \Vert y_i-z_i\Vert\leq \varepsilon$, from where $\diam(S(C_\mathcal U, (x_i^*)_\mathcal U,\alpha))\leq \varepsilon$ and the proof is finished.
\end{proof}

In the particular case of points of the form $\J(x)$, we can say more. 

\begin{theo}\label{dent_j(x)}
Let $C$ be a bounded convex subset of a Banach space $X$, $x\in C$ and $\mathcal U$ be a free ultrafilter on an infinite set $I$. Then $x\in\dent(C)$ if and only if $\J(x)\in\dent_{(X^*)_\mathcal U}(C_\mathcal U)$.
\end{theo}

\begin{proof}
If $x\in\dent(A)$, then $\J(x)\in\dent_{(X^*)_\mathcal U}(C_\mathcal U)$ by Theorem~\ref{theo:condisufidenting}. 

Now suppose $\J(x)\in\dent_{(X^*)_\mathcal U}(A)$. Let $\varepsilon>0$. There exist $(x_i^*)_\mathcal U\in (X^*)_{\mathcal U}$ and $\alpha>0$ such that $\diam(S(C_\mathcal U, (x_i^*)_\mathcal U,\alpha))\leq\varepsilon$ and $\J(x)\in S(C_\mathcal U, (x_i^*)_\mathcal U,\alpha)$, i.e. $\lim_\mathcal U x^*_i(x)>\sup_{C_\mathcal U} (x_i^*)_\mathcal U-\alpha$. Let $\eta\in(0,\alpha)$ such that $$\lim_\mathcal U x^*_i(x)>\sup_{C_\mathcal U} (x_i^*)_\mathcal U-\alpha+2\eta.$$ Define $$J_1=\{j\in I\ |\ x^*_j(x)>\sup_{C_\mathcal U} (x_i^*)_\mathcal U-\alpha+2\eta\}\in\mathcal U$$ and $$J_2=\{j\in I\ |\ \sup_{C_\mathcal U} (x_i^*)_\mathcal U>\sup_Cx^*_j-\eta\}\in\mathcal U.$$ Note that $J_2\in\mathcal U$ by Lemma~\ref{sup}. Furthermore, for all $j\in J:=J_1\cap J_2$, we have that $x\in S(C,x_j^*,\alpha-\eta)$. In fact, if $j\in J$, we have that $$x^*_j(x)>\sup_{C_\mathcal U} (x_i^*)_\mathcal U-\alpha+2\eta>\sup_Cx^*_j-\alpha+\eta.$$ Now let us show that there exists $j\in J$ such that $\text{diam}(S(C,x_j^*,\alpha-\eta))\leq 2\varepsilon$. Suppose by contradiction that it is not true. Then for all $j\in J$ there exist $y_j,z_j\in S(C,x_j^*,\alpha-\eta)$ such that $\|y_j-z_j\|>2\varepsilon$. We have that $$x^*_j(y_j)>\sup_Cx^*_j-\alpha+\eta$$ for all $j\in J$. It follows that $$\langle(x_i^*)_\mathcal U,(y_i)_\mathcal U\rangle\geq\lim_\mathcal U\sup_Cx_i^*-\alpha+\eta=\sup_{C_\mathcal U} (x_i^*)_\mathcal U-\alpha+\eta>\sup_{C_\mathcal U} (x_i^*)_\mathcal U-\alpha$$ by Lemma~\ref{sup}. This proves that $(y_i)_\mathcal U\in S(C_\mathcal U, (x_i^*)_\mathcal U,\alpha)$. In a similar way, we have that $(z_i)_\mathcal U\in S(C_\mathcal U, (x_i^*)_\mathcal U,\alpha)$. We deduce that $\|(y_i)_\mathcal U-(z_i)_\mathcal U\|\leq\varepsilon$, which contradicts the choice of $y_j$ and $z_j$.
\end{proof}

    \subsection{Exposed and strongly exposed points}

Let us conclude the section of general results with an analysis of strongly exposed points. As it is done in the previous subsection, we will begin by considering a uniform notion. 

\begin{defi}
Let $C$ be a bounded convex subset of a Banach space $X$. A set $\{x_i\}_{i\in I}\subset C$ is said to be a \textit{uniformly strongly exposed set} if there exists $\{x_i^*\}_{i\in I}\subset S_{X^*}$ such that for every $\varepsilon>0$ there exists $\alpha_\varepsilon>0$ satisfying that $$\begin{array}{ccc}
x_i\in S(C,x_i^*,\alpha_\varepsilon) & \mbox{ and } & \diam(S(C,x_i^*,\alpha_\varepsilon))<\varepsilon\quad \forall i\in I.
\end{array}$$
\end{defi}

This definition was probably introduced in the celebrated paper of J.~Lindenstrauss \cite{linds1963}, where it is proved that if a Banach space $X$ satisfies that its unit ball is the closed convex hull of a strongly exposed set then $X$ has Lindenstrauss property A, i.e. the set of norm-attaining operators $NA(X,Y)$ is dense in $L(X,Y)$ for every Banach space $Y$ \cite[Proposition~1]{linds1963}. See \cite[Section~3]{ccgmr} for a number of examples of Banach spaces where the previous condition holds.

Anyway, our interest in uniformly strongly exposed sets comes from the following result.

\begin{theo}\label{theo:carauniexposed}
Let $C$ be a bounded convex subset of a Banach space $X$ and $I$ be an infinite set. Let $\{x_i\}_{i\in I}$ be a family of points exposed in $C$ by $\{x_i^*\}_{i\in I}\subset B_{X^*}$. The following are equivalent:
\begin{enumerate}
    \item[(i)] $\{x^*_i\}_{i\in I}$ uniformly strongly exposes $\{x_i\}_{i\in I}$;
    \item[(ii)] $(x_i^*)_\mathcal U$ strongly exposes $(x_i)_\mathcal U$ in $C_\mathcal U$ for every free ultrafilter $\mathcal U$ on $I$;
    \item[(iii)] $(x_i^*)_\mathcal U$ exposes $(x_i)_\mathcal U$ in $C_\mathcal U$ for every free ultrafilter $\mathcal U$ on $I$.
\end{enumerate}
\end{theo}

\begin{proof}(i)$\Rightarrow$(ii). Note that,  by Lemma~\ref{sup}, $$\langle(x_i^*)_\mathcal U,(x_i)_\mathcal U\rangle=\lim_\mathcal Ux_i^*(x_i)=\lim_\mathcal U\sup_Cx_i^*=\sup_{C_\mathcal U} (x_i^*)_\mathcal U.$$
Now let $\varepsilon>0$ and take $\alpha>0$ given by the definition of uniformly strongly exposed set. Suppose that $(y_i)_\mathcal U\in S(C_\mathcal U, (x_i^*)_{\mathcal U}, \alpha)$. By the previous equalities, it means that $\lim_\mathcal Ux_i^*(y_i)>\lim_\mathcal Ux_i^*(x_i)-\alpha$. In particular, $$J:=\{i\in I\ |\ x_i^*(y_i)>x_i^*(x_i)-\alpha\}\in\mathcal U.$$ Then $\|y_i-x_i\|<\varepsilon$ for all $i\in J$. We conclude that $\|(x_i)_\mathcal U-(y_i)_\mathcal U\|\leq\varepsilon$, thus $(x_i)_\mathcal U\in\strexp_{(X^*)_\mathcal U}(C_\mathcal U)$.

(ii)$\Rightarrow$(iii) is obvious, so let us prove (iii)$\Rightarrow$(i). Assume that $\{x_i\}_{i\in I}$ is not uniformly strongly exposed by $\{x_i^*\}_{i\in I}$, and let us find a free ultrafilter $\mathcal U$ on $I$ so that $(x_i)_\mathcal U$ is not exposed by $(x_i^*)_\mathcal U$.

Since $\{x_i\}_{i\in I}$ is not uniformly strongly exposed by $\{x_i^*\}_{i\in I}$, there exists $\varepsilon_0>0$ so that, for every $n\in\mathbb N$, there exists $i_n\in I$ and $y_{i_n}\in C$ so that $x_{i_n}^*(y_{i_n})>\sup_C x_i^*-\frac{1}{n}$ but $\Vert x_{i_n}-y_{i_n}\Vert\geq \varepsilon_0$. Define the set $L:=\{i_n: n\in\mathbb N\}$ and note that $L$ is infinite. Otherwise, there is $n_0$ such that $i_n=i_{n_0}$ for $n\geq n_0$. We have that 
$$x^*_{i_{n_0}}(y_{i_{n_0}})=x^*_{i_{n}}(y_{i_{n}})>\sup_C x^*_{i_n}-\frac{1}{n}=\sup_C x^*_{i_{n_0}}-\frac{1}{n} \quad \forall n\geq n_0,$$ so taking limit we deduce that $x^*_{i_{n_0}}(y_{i_{n_0}})=\sup_C x^*_{i_{n_0}}$.
 Since $\|x_{i_{n_0}}-y_{i_{n_0}}\|\geq\varepsilon_0$, we derive a contradiction with the fact that $x_{i_{n_0}}^*$ exposes $x_{i_{n_0}}$. Consequently, $L$ is infinite. 
 
 Now, let $\mathcal U$ be a free ultrafilter on $I$ with $L\in \mathcal U$. Define $y_i:=y_{i_n}$ if $i=i_n$ and $y_i=0$ otherwise. First, note that \[ \Vert (x_i)_\mathcal U-(y_i)_\mathcal U\Vert =\lim_\mathcal U\Vert x_i-y_i\Vert\geq \varepsilon_0\]
 since the set $\{i\in I: \Vert x_i-y_i\Vert\geq \varepsilon_0\}$ belongs to $\mathcal U$. 

On the other hand, we claim that $\<(x_i^*)_{\mathcal U}, (y_i)_{\mathcal U}\>=\sup_{C_\mathcal U} (x_i^*)_{\mathcal U}$. Indeed, Lemma~\ref{sup} implies that
$$\sup_{C_\mathcal U} (x^*_i)_\mathcal U=\lim_\mathcal U \sup_C x_i^*=\lim_\mathcal U x_i^*(x_i),$$
so let us prove that $\langle(x_i^*)_\mathcal U,(y_i)_\mathcal U\rangle\geq \lim_\mathcal U x_i^*(x_i)$. To this end, pick $\varepsilon>0$. The set
$$B:=\{j\in I: \vert x_j^*(y_j)-\lim_\mathcal U x_i^*(y_i)\vert<\varepsilon\}$$
belongs to $\mathcal U$. Now, given $p\in\mathbb N$, find $$j\in B\cap L\cap \{k\in I: \vert x_k^*(x_k)-\lim_\mathcal U x_i^*(x_i)\vert<\varepsilon\}\setminus \{i_1,\ldots, i_p\}$$ (the previous set is non-empty because, actually, it belongs to $\mathcal U$). Now we have
$$\lim_\mathcal U x_i^*(y_i)>x_j^*(y_j)-\varepsilon>x_j^*(x_j)-\frac{1}{p}-\varepsilon>\lim_\mathcal U x_i^*(x_i)-\frac{1}{p}-2\varepsilon.$$
The arbitrariness of $p$ and $\varepsilon$ conclude that $\langle(x_i^*)_\mathcal U,(y_i)_\mathcal U\rangle=\sup_{C_\mathcal U} (x_i)_\mathcal U$. This shows that $(x_i)_{\mathcal U}$ is not exposed by $(x_i^*)_{\mathcal U}$ and finishes the proof.
\end{proof}

Now, we focus on the case of elements of the form $\J(x)$ for $x\in C$.    
    
\begin{coro}\label{improvement}
Let $C$ be a bounded convex subset of a Banach space $X$ and $\mathcal U$ be a free ultrafilter on an infinite set $I$. Let $x\in C$. Then
\begin{enumerate}
    \item[(a)] if $\J(x)\in\exp_{(X^*)_\mathcal U}(C_\mathcal U)$, then $x\in\exp(C)$;
    \item[(b)] $\J(x)\in\strexp_{(X^*)_\mathcal U}(C_\mathcal U)$ if and only if $x\in\strexp(C)$.
\end{enumerate}
\end{coro}

\begin{proof}
$(a)$ Suppose that $\J(x)$ is exposed by $(x^*_i)_\mathcal U$. By weak*-compactness of $B_{X^*}$, define $x^*=w^*$-$\lim_\mathcal Ux^*_i$.  Let $y\in A$ such that $y\neq x$. We have  $$x^*(x)=\lim_\mathcal Ux^*_i(x)=\langle (x^*_i)_\mathcal U, \J(x)\rangle>\langle (x^*_i)_\mathcal U, \J(y)\rangle=\lim_\mathcal Ux^*_i(y)=x^*(y),$$ i.e. $x^*$ exposes $x$.

$(b)$ Follows directly from Theorem~\ref{theo:carauniexposed}. 
\end{proof}

If $\mathcal U$ is supposed to be CI, the previous corollary can be improved:

\begin{thm}\label{th:jstrexp}
Let $C$ be a bounded convex subset of a Banach space $X$, $\mathcal U$ be a CI ultrafilter on an infinite set $I$ and let $x\in C$. The following assertions are equivalent:
\begin{enumerate}
    \item[(i)] $\J(x)\in\strexp_{(X^*)_\mathcal U}(C_\mathcal U)$; 
    \item[(ii)] $\J(x)\in\exp_{(X^*)_\mathcal U}(C_\mathcal U)$;
    \item[(iii)] $x\in\strexp(C)$.
\end{enumerate}
\end{thm}
    
\begin{proof}
In view of the previous corollary, we only have to prove that if $\J(x)\in\exp_{(X^*)_\mathcal U}(C_\mathcal U)$ then $x\in\strexp(C)$. Let us suppose that $\J(x)$ is exposed by $(x^*_i)_\mathcal U$. By weak*-compactness of $B_{X^*}$, define $x^*=w^*$-$\lim_\mathcal Ux^*_i$. %We need to prove that $x$ is strongly exposed by $x^*$. 
The proof of $(a)$ in the previous corollary shows that $x^*$ exposes $x$, we will show that actually $x^*$ strongly exposes $x$. Consider a sequence $(x_n)_n\subset C$ such that $x^*(x_n)\to x^*(x)$ and suppose by contradiction that $(x_n)_n$ does not converge to $x$. Then there exist a subsequence $(x_{n_k})_k$ of $(x_n)_n$ and $\beta>0$ such that $\|x_{n_k}-x\|>\beta$ for all $k\in\mathbb N$. Note that $$|x^*(x)-\lim_{\mathcal U,i}x^*_i(x_{n_k})|=|x^*(x)-x^*(x_{n_k})|\xrightarrow[k]{} 0,$$ so by considering a further subsequence if necessary, we can also suppose that $|x^*(x)-\lim_{\mathcal U,i}x^*_i(x_{n_k})|<\frac{1}{k}$ for all $k\geq 1$. For all $n\geq 1$, we define $$A_n=\left\{i\in I\ :\ |x^*(x)-x^*_i(x_{n_k})|<\frac{1}{k}\ \forall k\leq n\right\}\cap I_n\in\mathcal U$$ where $(I_n)_n$ is a decreasing sequence of sets in $\mathcal U$ such that $\bigcap_nI_n=\emptyset$. Now we define $(y_i)_\mathcal U\in C_\mathcal U$ by $y_i=x_{n_k}$ if $k\in A_k\setminus A_{k+1}$ for some $k\geq 1$ and $y_i=0$ of $i\in I\setminus A_1$. One can verify that  $\lim_{\mathcal U, i} x^*_i(y_i)=x^*(x)$,
which means exactly that $\langle (x^*_i)_\mathcal U, (y_i)_\mathcal U\rangle=\langle (x_i^*)_{\mathcal U}, \J(x) \rangle$.
%$\langle (x^*_i)_\mathcal U, (y_i)_\mathcal U\rangle=\langle j(x), (y_i)_\mathcal U\rangle$.
Since $\J(x)$ is exposed by $(x^*_i)_\mathcal U$, we deduce that $\J(x)=(y_i)_\mathcal U$. We obtain that $\lim_{\mathcal U, i}\|x-y_i\|=0$, but it is easily seen that it contradicts the fact that $\|x_{n_k}-x\|>\beta$ for all $k\in\mathbb N$.
\end{proof}

If $C$ is a bounded convex set, we sum up the properties linking $x$ and $\J(x)$ in the following graph of implications:
\vspace{0.2cm}

\begin{center}
\adjustbox{scale=0.9,center}{%
\begin{tikzcd}[arrows=Rightarrow]
\text{$x\in\strexp(C)$} \arrow[r] \arrow[d,Leftrightarrow] & \text{$x\in\dent(C)$} \arrow[r] \arrow[d,Leftrightarrow] & \text{$x\in\strext(C)$} \arrow[r] \arrow[d,Leftrightarrow] & \text{$x\in\ext(C)$}  \\
\text{$\J(x)\in\strexp_{(X^*)_\mathcal U}(C_\mathcal U)$} \arrow[r] & \text{$\J(x)\in\dent_{(X^*)_\mathcal U}(C_\mathcal U)$} \arrow[r] & \text{$\J(x)\in\strext(C_\mathcal U)$} \arrow[r,Leftrightarrow] & \text{$\J(x)\in\ext(C_\mathcal U)$} \arrow[u]\\
\end{tikzcd}}
\end{center}

Note that none of the previous implications can be reversed in the general case (since there exist extreme points which are not strongly extreme, strongly extreme points which are not denting and denting points which are not strongly exposed).

\begin{thm}\label{theo:expequistronglyCI}
Let $C$ be a bounded convex subset of a Banach space $X$, and $\mathcal U$ be a CI ultrafilter on an infinite set $I$. Then %$\exp_{(X^*)_\mathcal U}(C_\mathcal U)=\strexp_{(X^*)_\mathcal U}(C_\mathcal U)$. 
$\exp_{(X^*)_\mathcal U}(C_\mathcal U)\subset\strexp(C_\mathcal U)$.
\end{thm}

\begin{proof}
Let $(U_n)_n\subset\mathcal U$ be a strictly decreasing sequence of sets such that $\bigcap_{n>0}U_n=\emptyset$. Let $(x_i)_\mathcal U\in\exp_{(X^*)_\mathcal U}(C_\mathcal U)$ be exposed by $(x^*_i)_\mathcal U\in (X^*)_\mathcal U$. We need to show that $(x_i)_\mathcal U\in\strexp(C_\mathcal U)$. By Theorem \ref{th:jstrexp}, it is enough to prove that $\J((x_i)_\mathcal U)\in\exp_{(X^*)_{\mathcal U\times\mathcal U}}(C_{\mathcal U\times\mathcal U})$. Define $(x^*_{i,j})_{\mathcal U\times\mathcal U}\in(X^*)_{\mathcal U\times\mathcal U}$ by $x^*_{i,j}=x^*_i$ for all $i,j\in I$ and let us prove that $(x^*_{i,j})_{\mathcal U\times\mathcal U}$ exposes $\J((x_i)_\mathcal U)$.

First note that, using Lemma \ref{limit}, we have that $$\langle (x^*_{i,j})_{\mathcal U\times\mathcal U},\J((x_i)_\mathcal U)\rangle=\lim_{\mathcal U\times\mathcal U}x^*_i(x_i)=\lim_\mathcal U x^*_i(x_i)=\langle (x^*_i)_\mathcal U, (x_i)_\mathcal U\rangle.$$ Take $(y_{i,j})_{\mathcal U\times\mathcal U}\in C_{\mathcal U\times\mathcal U}$ arbitrary. Using again Lemma \ref{limit} and the fact $(x_i)_\mathcal U$ is exposed by $(x^*_i)_\mathcal U$, we obtain that
\begin{align*}
\langle (x^*_{i,j})_{\mathcal U\times\mathcal U},(y_{i,j})_{\mathcal U\times\mathcal U}\rangle&=\lim_{\mathcal U,j}\lim_{\mathcal U,i}\langle x^*_i,y_{i,j}\rangle =\lim_{\mathcal U,j}\langle (x^*_i)_\mathcal U, (y_{i,j})_{\mathcal U,i}\rangle\\
&\leq \langle (x^*_i)_\mathcal U,(x_i)_\mathcal U\rangle=\langle (x^*_{i,j})_{\mathcal U\times\mathcal U},\J((x_i)_\mathcal U)\rangle,
\end{align*}
proving that $(x^*_{i,j})_{\mathcal U\times\mathcal U}$ reaches his maximum on $C_{\mathcal U\times\mathcal U}$ at $\J((x_i)_\mathcal U)$.

Now suppose by contradiction that there exists $(y_{i,j})_{\mathcal U\times\mathcal U}\in C_{\mathcal U\times\mathcal U}$ such that $(y_{i,j})_{\mathcal U\times\mathcal U}\neq \J((x_i)_\mathcal U)$ and 
$$\langle (x^*_{i,j})_{\mathcal U\times\mathcal U},(y_{i,j})_{\mathcal U\times\mathcal U}\rangle=\langle (x^*_{i,j})_{\mathcal U\times\mathcal U}, \J((x_i)_\mathcal U)\rangle.$$ Let $\varepsilon_0>0$ such that $\|(y_{i,j})_{\mathcal U\times\mathcal U}-\J((x_i)_\mathcal U)\|>\varepsilon_0$. We can also assume that $\| y_{i,j}-x_i\|>\varepsilon_0$ holds for every $i,j\in I$. We have that $$\left\{(i,j)\in I^2\ |\ |x^*_i(y_{i,j})-x^*_i(x_i)|<\frac{1}{n}\right\}\in\mathcal U\times\mathcal U$$ for all $n>0$, that is $$J_n:=\left\{j\in I\ | \left\{i\in I\ |\ |x^*_i(y_{i,j})-x^*_i(x_i)|<\frac{1}{n}\right\}\in\mathcal U\right\}\in\mathcal U.$$ For $j\in I$ and $n>0$, define $I_{n,j}=\left\{i\in I\ |\ |x^*_i(y_{i,j})-x^*_i(x_i)|<\frac{1}{n}\right\}$. Following verbatim the last steps in the proof of Theorem \ref{theo:extequistronglyCI}, we obtain a contradiction.
\end{proof}

It is natural to think that indeed the inclusion $\exp_{(X^*)_\mathcal U}(C_\mathcal U)\subset\strexp_{(X^*)_\mathcal U}(C_\mathcal U)$ holds (obtaining clearly the equality of the above sets). However, our techniques in Theorem \ref{th:jstrexp} does not allow us to conclude it.

\section{Extremality and compactness}\label{seccion:particulares}

In this section we will study the extremality under compactness assumptions. To be more precise, let $X$ be a Banach space and $K\subseteq B_X$ be a convex bounded subset. We will deal with the assumption that $K_\mathcal U$ is weakly compact (see below the definition of super-weakly compact set). Before we enter in details, let us explain why this context, though very restrictive, is interesting for us: looking to our results for denting points and strongly exposed points, we have not been able to completely characterise when a point $(x_i)_\mathcal U$ is a denting (respectively strongly exposed) point because we do not have a good access to the space $(X_\mathcal U)^*$, which differs from $(X^*)_\mathcal U$ if $X$ is not superreflexive.

However, in the particular case of $K_\mathcal U$ being weakly compact this difficulty is overcome thanks to Lemma~\ref{lemma:dentingsubnormante}. For instance, here we obtain that a point $(x_i)_\mathcal U\in K_\mathcal U$ is denting if, and only if, $(x_i)_\mathcal U$ belongs to a sequence of slices of diameter as small as desired where the slices are defined by elements of $(X^*)_\mathcal U$. This difficulty will be overcome in the context of super weakly compact subsets.

Let $K$ be a bounded subset of a Banach space $X$. We say that $K$ is relatively super-weakly compact if $K_\mathcal U$ is relatively weakly compact for all free ultrafilters $\mathcal U$ on an arbitrary infinite set $I$. If furthermore $K$ is weakly-closed, we say that $K$ is super weakly compact. Note that if $K$ is a closed convex symmetric super-weakly compact set then $K$ is the unit ball of the superreflexive space $Y=(\text{span}(K),|\cdot|_K)$ where $|\cdot|_K$ is the Minkowski functional of $K$.

%\vspace{0.5cm}

The roots of super weak compactness can be traced back to the work of B.~Beauzamy in the 70's \cite{Beauzamy}. He introduced the notion of \textit{uniformly convex operator}, which turns to be equivalent (under renorming) to the one of \textit{super-weakly compact operator} (see \cite{Heinrich}). In   \cite{dentable},  M.~Raja considered the notion of \textit{finite dentability}, which coincides with the one of super weak  compactness for closed convex bounded sets. The current terminology is introduced in \cite{SWC1}, which also contains a characterization in terms of finite representability of sets that reinforces the parallelism with super-reflexive Banach spaces. Super-weak compactness has became a prolific field of investigation, see e.g. \cite{ChengOtro, LancienRaja, RajaSWCG, Tu}.

%\vspace{0.5cm}    
    
Recall that every weakly compact convex set is the closed convex hull of its strongly exposed points (see Theorem~8.13 in \cite{BST} for example). In the case of ultrapowers we can say a bit more.

\begin{theo}\label{strexp}
Let $K\subset X$ be a relatively super weakly compact convex set and $\mathcal U$ be a CI ultrafilter on an infinite set $I$. Then $$K_\mathcal U=\overline{\co}(\strexp_{(X^*)_\mathcal U}(K_\mathcal U)).$$
\end{theo}

This is just a particular case of the following result:

\begin{lema}\label{lemma:dualpair} Let $K$ be weakly compact convex subset of a Banach space $X$ and $Z$ be a subspace of $X^*$ such that $(X,Z)$ is a dual pair. Then
\[ K=\cconv(\strexp_Z(K)).\]
\end{lema}

\begin{proof}
First, let's show that every subset $C$ of $K$ is $Z$-dentable, that is, there are slices of $C$ given by elements of $Z$ with arbitrarily small diameter. If $\overline{\co}(C)$ is $Z$-dentable then $C$ is dentable too, so we can suppose that $C$ is closed and convex. Since $K$ is weakly compact, so is $C$. In particular, $C$ is dentable and then $Z$-dentable by Lemma~\ref{lemma:dentingsubnormante}. A slight modification of Theorem~8 in \cite{Bourgain} allows us to conclude that the subset of $Z$
that strongly exposes elements of $K$ is dense in $Z$. Now, suppose by contradiction that $K\neq\overline{\co}(\strexp_{Z}(K)$. Since $\overline{\co}(\strexp_{Z}(K))$ is weakly compact and then $\sigma(X,Z)$-compact, there exists $x^*\in Z$ such that $$\sup_{K} x^*>\sup_{\strexp_{Z}(K)}x^*.$$ By density, we deduce that there exists $y^*\in Z$ strongly exposing an element $x\in K$ such that $$y^*(x)=\sup_{K} y^*>\sup_{\strexp_{Z}(K)}y^*,$$ which is a contradiction.
\end{proof}

\begin{proof}[Proof of Theorem~\ref{strexp}] Note that $K_\mathcal U$ is weakly compact thanks to Proposition~\ref{fact:closed}. Now, apply Lemma~\ref{lemma:dualpair} taking $Z=(X_\mathcal U)^*\subset (X_\mathcal U)^*$.
\end{proof}

K.~Tu has recently proved that in \cite{Tu} that the closed convex hull of a relatively super weakly compact set is super weakly compact. More precisely, he obtained that
$$\overline{(\co(A))_{\mathcal U}}= \overline{\co}(A_{\mathcal U}).$$ for any relatively super weakly compact set $A$. Using this result, it is possible to localise the set of extreme points of a super weakly compact set:

\begin{prop}\label{extremepoint}
Let $K\subset X$ be a super weakly compact convex set and $\mathcal U$ be a CI ultrafilter on an infinite set $I$. Then 
\[\ext(K_\mathcal U)\subset\overline{(\strexp(K))_\mathcal U}^w \quad \text{and}\quad \dent(K_\mathcal U)\subset (\strexp K)_{\mathcal U}.\]
\end{prop}

\begin{proof}
Since $K$ is weakly compact, we have that $K=\cconv{(\strexp(K))}$. Thus, using K.~Tu result, it follows that $$K_\mathcal U=(\overline{\co}(\strexp(K)))_\mathcal U=\overline{(\co(\strexp(K)))_{\mathcal U}}=\overline{\co}((\strexp(K))_\mathcal U).$$ 
so any slice of $K_{\mathcal U}$ has non-empty intersection with $(\strexp(K))_{\mathcal U}$. Since the slices are a neighbourhood basis for the extreme (resp. denting) points of $K_{\mathcal U}$ in the weak (resp. norm) topology, we have $\ext(K_\mathcal U)\subset\overline{(\strexp(K))_\mathcal U}^w$ and $\dent(K_\mathcal U)\subset \overline{(\strexp K)_{\mathcal U}}=(\strexp K)_{\mathcal U}$, where the last equality follows from Lemma~\ref{fact:closed}. 
\end{proof}

\begin{prop}\label{dent}
Let $K$ be a super weakly compact subset of a Banach space $X$, $\mathcal U$ be a CI ultrafilter on an infinite set $I$ and let $x\in K$. Then $\J(x)\in\dent(K_\mathcal U)$ if and only if $x\in\dent(K)$.
\end{prop}

\begin{proof}
Since $K$ is super weakly compact, it follows that $\J(x)$ is $(X^*)_\mathcal U$-denting if and only if $\J(x)\in\dent(K_\mathcal U)$ (by Lemma~\ref{lemma:dentingsubnormante}). We conclude by Theorem~\ref{dent_j(x)}.
\end{proof}

Note that Theorem \ref{strexp} is a useful tool in the search of a characterisation of when $(x_i)_\mathcal U$ is a denting point in $K_\mathcal U$ if $K$ is super weakly compact. However, in order to get a complete characterisation in terms of a condition on the points $x_i$'s, we will consider a notion which is stronger than super weak compactness: the one of \textit{uniform convexity} (see Definition~\ref{def:uniconvex}). This geometric property on $K$ will allow us to characterise the denting points of uniformly convex subsets of a Banach space (see Theorem \ref{theo:dentinguconvex}). 

One can think that there is a big difference between super weakly compact sets and uniformly convex sets. However, thanks to a result of M. Raja and G. Lancien (see \cite[Proposition~4.3]{LancienRaja}), we see that from a topological point of view this is not the case. Indeed, given a symmetric super weakly compact subset $K$ and $\varepsilon>0$, there exists a uniformly convex set $C_\varepsilon$ so that $C_\varepsilon\subseteq K\subseteq (1+\varepsilon)C_\varepsilon$. This result should be compared with a classical theorem of Enflo (see for instance \cite[Theorem 9.14]{BST}) which says that, given a superreflexive Banach space $(X,\Vert \cdot\Vert)$ then, for every $\varepsilon>0$, there exists a renorming $\vert\cdot\vert_\varepsilon$ on $X$ so that $(X,\vert\cdot\vert_\varepsilon)$ is a uniformly convex Banach space and so that 
$$\frac{1}{1+\varepsilon}\vert x\vert_\varepsilon\leq \Vert x\Vert\leq (1+\varepsilon)\vert x\vert_\varepsilon\qquad \forall x\in X.$$

A localised version of Enflo's theorem has been established in \cite{dentable} where the author proved that if $K$ is a super-weakly compact convex subset of a Banach space $X$, then there exists a equivalent norm $|\cdot|$ on $X$ such that the restriction of $|\cdot|^2$ to $K$ is uniformly convex. It is worth noting that the spaces which are generated by a super-weakly compact set have been characterized in terms of strongly uniformly Gâteaux renorming (see \cite{RajaSWCG}).

Let us now consider the formal definition of uniformly convex set.

\begin{defi}\label{def:uniconvex}
A symmetric bounded closed convex set $C$ of a Banach space $X$ is said to be uniformly convex if for every $\varepsilon>0$ there exists $\delta>0$ such that $$\forall x,y\in C,\ \ \|x-y\|>\varepsilon\ \implies\ \frac{x+y}{2}\in(1-\delta)C.$$ In such case, we define the convexity modulus of $C$ by $$\delta_C(\varepsilon)=\inf\left\{1-\left|\frac{x+y}{2}\right|_C\ :\ x,y\in C,  \|x-y\|>\varepsilon\right\}$$ where $|\cdot|_{C}$ is the Minkowski functional of $C$. By convention, $\inf\emptyset=1$.
\end{defi}

Raja proved that a closed convex bounded subset is finitely dentable (see the introduction in \cite{dentable} for the definition)  if and only if it admits a uniformly convex function (see Theorem~2.2 in \cite{dentable}). Moreover, it is possible to construct directly a uniformly convex function on any super weakly convex set using the fact that such a set does not admit dyadic separated trees of arbitrary height (see Theorem~4.4 in \cite{tree}). It follows that a closed convex set is super weakly compact if and only if it is finitely dentable. Since a uniformly convex set is finitely dentable, the authors of \cite{LancienRaja} obtained that any uniformly convex set is super weakly compact \cite[Proposition~4.2]{LancienRaja}. 

We will also consider the following weakening of uniform convexity.

\begin{defi}
A symmetric bounded closed convex set $C$ of a Banach space $X$ is said to be strictly convex if for all $x,y\in C$ such that $x\neq y$ and $|x|_C=|y|_C=1$, one has that $\left|\frac{x+y}{2}\right|_C<1$.
\end{defi}

In general, every extreme point $x$ of a symmetric bounded closed convex set $C\neq\{0\}$ satisfies $|x|_C=1$. In the case $C$ is strictly convex, one can easily check that indeed $\ext(C)=|\cdot|_C^{-1}(\{1\})$. We will use this fact in the sequel. 

The following result generalizes the fact that a Banach space is uniformly convex if and only if its ultrapower is strictly (or uniformly) convex.

\begin{prop}\label{unifconvex}
Let $K$ be a symmetric bounded convex subset of a Banach space $X$. Let $\mathcal U$ be a CI ultrafilter on an infinite set $I$. The following assertions are equivalent:
\begin{enumerate}
    \item[(i)] $K$ is uniformly convex;
    \item[(ii)] $K_\mathcal U$ is uniformly convex.
\end{enumerate}
In that case, we have that $\delta_K=\delta_{K_\mathcal U}$. Moreover, if $0$ is an interior point of $K$, then the previous statements are equivalent to:
 \begin{enumerate}
    \item[(iii)] $K_\mathcal U$ is strictly convex. 
\end{enumerate}
\end{prop}

For the proof we will need the following result.

\begin{lema}\label{minkovski}
Let $C$ be a symmetric bounded convex subset of a Banach space $X$, $\mathcal U$ be a CI ultrafilter on an infinite set $I$ and $(x_i)_\mathcal U\in C_{\mathcal U}$. Then $|(x_i)_\mathcal U|_{C_\mathcal U}\leq\lim_\mathcal U|x_i|_C$. Moreover, if $0$ is an interior point of $C$, the reverse equality also holds. 
\end{lema}

\begin{proof}
Define $l:=\lim_\mathcal U|x_i|_C$. Let $\varepsilon>0$. For all $i\in I$, we have that $x_i\in(|x_i|_C+\varepsilon)C$. Define $J:=\{i\in I: ||x_i|_C-l|<\varepsilon\}\in\mathcal U$. For all $i\in J$, it follows that $x_i\in(l+2\varepsilon)C$. Then $(x_i)_\mathcal U\in (l+2\varepsilon)C_\mathcal U$ for all $\varepsilon>0$. Since $C_\mathcal U$ is closed, we deduce that $(x_i)_\mathcal U\in lC_\mathcal U$. We conclude that $|(x_i)_\mathcal U|_{C_\mathcal U}\leq l$.

Now, assume that $0$ is an interior point of $C$. Let $\lambda:=\vert (x_i)_\mathcal U\vert_{C_\mathcal U}$ and notice that $(x_i)_\mathcal U\in \lambda C_\mathcal U$, so there exists $(y_i)_{i\in I}\in C^I$ so that $(x_i)_\mathcal U=(\lambda y_i)_\mathcal U$. Note that
$$\lim_\mathcal U \Vert x_i-\lambda y_i\Vert=0\Leftrightarrow \lim_\mathcal U \vert x_i-\lambda y_i\vert_C=0$$
since $\Vert\cdot\Vert$ and $\vert \cdot\vert_C$ are equivalent norms on $\text{span}(C)$. This implies that
$$\lim_\mathcal U \vert x_i \vert_C=\lim_\mathcal U\vert\lambda y_i\vert_C=\lambda\lim_\mathcal U\vert y_i\vert_C\leq \lambda$$
where the last inequality holds since $y_i\in C$ holds for every $i$. This proves the equality in such case.
\end{proof}

\begin{proof}[Proof of Proposition \ref{unifconvex}]
$(i)\implies (ii)$ Suppose that $K$ is uniformly convex. Note that $K_{\mathcal U}$ is closed since $\mathcal U$ is CI. Let $\varepsilon>0$. Take $(x_i)_\mathcal U, (y_i)_\mathcal U\in K_\mathcal U$ such that $\|(x_i)_\mathcal U-(y_i)_\mathcal U\|>\varepsilon$. Then, we can suppose (changing some coordinates if necessary) that $\|x_i-y_i\|>\varepsilon$ for all $i\in I$. Let $\eta\in(0,\delta_K(\varepsilon))$ arbitrary. It follows that $\frac{x_i+y_i}{2}\in(1-\eta)K$ for all $i\in I$ and then $$\frac{(x_i)_\mathcal U+(y_i)_\mathcal U}{2}=\left(\frac{x_i+y_i}{2}\right)_\mathcal U\in (1-\eta)K_\mathcal U.$$ Since $\eta$ was arbitrary, we conclude that $0<\delta_{K}(\varepsilon)\leq \delta_{K_\mathcal U}(\varepsilon)$, i.e. $K_\mathcal U$ is uniformly convex. 

$(ii)\implies (i)$ Suppose that $K_\mathcal U$ is uniformly convex. Let $\varepsilon>0$. Let $x,y\in K$ such that $\|x-y\|>\varepsilon$. Let $\eta\in(0,\delta_{K_\mathcal U}(\varepsilon))$ arbitrary. It follows that $\|\J(x)-\J(y)\|>\varepsilon$ and then $j\left(\frac{x+y}{2}\right)=\frac{\J(x)+\J(y)}{2}\in(1-\eta)K_\mathcal U$. Let $(z_i)_{i\in I}\in K^I$ such that $j\left(\frac{x+y}{2}\right)=(1-\eta)(z_i)_\mathcal U$. Since $\lim_\mathcal U\left\|\frac{x+y}{2}-(1-\eta)z_i\right\|=0$, it follows that $\frac{x+y}{2}\in\overline{(1-\eta)K}=(1-\eta)K$. Since $\eta$ was arbitrary, we conclude that $0<\delta_{K_\mathcal U}(\varepsilon)\leq\delta_K(\varepsilon)$, i.e. $K$ is uniformly convex.

Now suppose that $0$ is an interior point of $K$ and that $\mathcal U$ is a CI ultrafilter. $(ii)\implies (iii)$ is obvious. We will show the implication $(iii)\implies (i)$. Let suppose that $K$ is not uniformly convex. Then there exists $\varepsilon>0$ such that for all $n\in\mathbb N$, there exist $x_n,y_n\in K$ with $\|x_n-y_n\|>\varepsilon$ and $\left|\frac{x_n+y_n}{2}\right|_K\to 1$.
Let $(I_n)_n\subset \mathcal U$ be a sequence of distinct sets such that $\bigcap_nI_n=\emptyset$, $I_0=I$ and $I_{n+1}\subset I_n$ for all $n\in\mathbb N$. Define $x'_i=x_n$ if $i\in I_n\setminus I_{n+1}$ for some $n\in\mathbb N$. Define $y'_i$ in the same way. It is clear that $\|(x'_i)_\mathcal U-(y'_i)_\mathcal U\|\geq\varepsilon$. Moreover, it is easy to show that $\lim_\mathcal U\left|\frac{x'_i+y'_i}{2}\right|_K= 1$. The previous lemma implies that $$\left|\frac{(x'_i)_\mathcal U+(y'_i)_\mathcal U}{2}\right|_{K_\mathcal U}=1.$$ By triangle inequality, we also have that $|(x'_i)_\mathcal U|_{K_\mathcal U}=|(y'_i)_\mathcal U|_{K_\mathcal U}=1$. So $K_\mathcal U$ cannot be strictly convex.
\end{proof}

\begin{rema}
Proposition~\ref{unifconvex} reproves the very well-known result that a Banach space $X$ is uniformly convex if, and only if, $X_{\mathcal U}$ is strictly convex, where $\mathcal U$ is a CI ultrafilter. 
\end{rema}

%\begin{prop}
%Let $X$ be a Banach space and $C$ be bounded closed convex subset and $\mathcal U$ be a free ultrafilter over $I$. Let $\{x_i\}_{i\in I}$ be a subset of $C$. If $(x_i)_\mathcal U$ is extreme in $C_\mathcal U$ then $\lim_\mathcal U |x_i|_C=1$.
%\end{prop}

%\begin{proof}
%If $\lim_\mathcal U |x_i|_C<1-\alpha$ for some $\alpha>0$, the set
%$$J:=\{i\in I: |x_i|_C<1-\alpha\}\in \mathcal U.$$
%Take $0<\beta<\min\{1,\alpha\}$ and define $y_i=(1+\beta)x_i$ if $i\in J$ and $y_i=x_i$ if $i\notin J$. Similarly, define $z_i=(1-\beta)x_i$ if $i\in J$ and $y_i=x_i$ if $i\notin J$. Note that $(y_i)_\mathcal U,(z_i)_\mathcal U\in C_\mathcal U$. Indeed, given $i\in I\setminus J$, then $y_i=z_i=x_i\in C$. Moreover, if $i\in J$, then $|y_i|_C=(1+\beta)|x_i|_C<(1+\beta)(1-\alpha)1+\beta-\alpha-\alpha\beta<1-\alpha\beta\leq 1$, from where $y_i\in C$. Morevoer, $|z_i|_C=(1-\beta)|x_i|_C=|x_i|_C-|x_i|_C\beta\leq |x_i|_C\leq 1$, from where $z_i\in C$ too. Moreover, $x_i=\frac{y_i+z_i}{2}$ for every $i\in I$, thus $(x_i)_\mathcal U=\frac{(y_i)_\mathcal U+(z_i)_\mathcal U}{2}$. On the other hand, observe that $\Vert x_i-y_i\Vert=\beta$ holds for every $i\in A$, from where $\Vert (x_i)_\mathcal U-(y_i)_\mathcal U\Vert=\lim_\mathcal U \Vert x_i-y_i\Vert\geq \beta>0$, so $(x_i)_\mathcal U$ is not an extreme point of $C_\mathcal U$.
%\end{proof}

In the sequel we aim to give a characterisation of the extreme points of a uniformly convex set. In order to do so, we need a preliminary result.  

\begin{lema}\label{lemma:limitext}
Let $C$ be symmetric bounded convex subset of a Banach space $X$ and $\mathcal U$ be a CI ultrafilter on an infinite set $I$. If $(x_i)_\mathcal U\in\ext(C_\mathcal U)$, then $\lim_\mathcal U |x_i|_C=1$.
\end{lema}

\begin{proof}
Since $x_i\in C$ for all $i\in I$, we have that $\lim_\mathcal U |x_i|_C\leq 1$. Moreover we have that $(x_i)_\mathcal U\in\ext(C_\mathcal U)$ so $|(x_i)_\mathcal U|_{C_\mathcal U}=1$. We conclude by Lemma~\ref{minkovski}.
\end{proof}

\begin{thm}\label{ext_ultrapower_unifconv}
Let $K$ be a uniformly convex subset of a Banach space $X$ and $\mathcal U$ be a CI ultrafilter on an infinite set $I$. Let $(x_i)_\mathcal U\in K_\mathcal U$. The following assertions are equivalent: 
\begin{enumerate}
    \item[(i)] $(x_i)_\mathcal U\in\ext(K_\mathcal U)$;
    \item[(ii)] for any $(y_i)_{i\in I}\in K^I$ such that $(y_i)_{\mathcal U}=(x_i)_{\mathcal U}$, it follows $\lim_\mathcal U |y_i|_K=1$.
\end{enumerate}
If $0$ is an interior point of $C$, then they are also equivalent to: 
\begin{enumerate}
    \item [(iii)] $\lim_{\mathcal U} |x_i|_K=1$.
\end{enumerate} 
\end{thm}

\begin{proof}
Suppose that $(x_i)_\mathcal U\in\ext(K_\mathcal U)$. Let $(y_i)_{i\in I}\in K^I$ such that $(y_i)_{\mathcal U}=(x_i)_{\mathcal U}$. Obviously, we have that $(y_i)_\mathcal U\in\ext(K_\mathcal U)$ and we conclude by the previous lemma.

Now suppose that $(x_i)_\mathcal U\notin\ext(K_\mathcal U)$. Then there exist $(y_i)_\mathcal U,(z_i)_\mathcal U$ such that $(x_i)_\mathcal U=\frac{(y_i)_\mathcal U+(z_i)_\mathcal U}{2}$ and $\|(y_i)_\mathcal U-(z_i)_\mathcal U\|>\varepsilon$ for some $\varepsilon>0$. Let $\delta$ associated to $\varepsilon$ given by the uniform convexity of $K_\mathcal U$. It follows that $(x_i)_\mathcal U=\frac{(y_i)_\mathcal U+(z_i)_\mathcal U}{2}\in (1-\delta)K_\mathcal U$. So there exists $(y'_i)_{i\in I}\in (1-\delta)K^I\subset K^I$ such that $(x_i)_\mathcal U=(y'_i)_\mathcal U$. Since $|y'_i|_K\leq 1-\delta$ for all $i\in I$, it follows that  $\lim_\mathcal U |y'_i|_K<1$.

Finally, assume that $0$ is an interior point of $C$. Clearly (ii)$\Rightarrow$(iii), and (iii)$\Rightarrow$(i) by Lemmas~\ref{minkovski} and the comment following the definition of strict convexity. 
\end{proof}

The next proposition shows the extremal structure of a uniformly convex set has great properties.

\begin{prop}\label{ext_unifconvex}
Let $K\subset X$ be a uniformly convex set of a Banach space $X$ and $Z$ be a subspace of $X^*$. Then
\begin{enumerate}
    \item[(a)] $\ext(K)=\dent(K)$.
    \item[(b)] $\displaystyle\exp_{Z}(K)=\strexp_Z(K) =\{x\in K\ | \exists x^*\in Z\ :\  \sup_Kx^*=x^*(x)>0 \}$.
\end{enumerate}
\end{prop}

\begin{proof}
(a) Let $x\in\ext(K)$ and let $\varepsilon>0$. Take $\delta$ associated to $\varepsilon$ given by the definition of uniform convexity. By the Hahn-Banach theorem, there exists $x^*\in X^*\setminus\{0\}$ such that $$x^*(x)>\sup_{(1-\delta)K}x^*.$$ Since $K$ is symmetric, we can suppose that $\sup_Kx^*=1$. Define a slice of $K$ by $S:=S(x^*,K,\delta)$. We have that $x^*(x)>\sup_{(1-\delta)K}x^*=1-\delta$, so $x\in S$ and $S\cap (1-\delta)K=\emptyset$. Let's show that $\text{diam}(S)\leq\varepsilon$. Suppose on the contrary that there exist $y,z\in S$ such that $\|y-z\|>\varepsilon$. By uniform convexity, it follows that $\frac{y+z}{2}\in (1-\delta)K$. This is a contradiction since $\frac{y+z}{2}\in S$. 

(b) Clearly $\strexp_Z(K)\subset \exp_Z(K)$. Now, assume that $x^*\in Z$ exposes $x$. Then $0=x^*(0)<x^*(x)$, so we get that $x^*$ satisfies our purposes.

Finally, that $x^*\in Z$ satisfies that $\sup_K x^*=x^*(x)>0$, and let us prove that $x^*\in Z$ strongly exposes $x$. To this end, pick $\varepsilon>0$, and let us find a slice of $K$ determined by $x^*$ with diameter smaller than $\varepsilon$. Let $\delta>0$ be associated to $\varepsilon$ in the definition of uniformly convex set. Without loss of generality, we can suppose that $x^*(x)=\sup_Kx^*=1$. Let $y\in K$ so that $\Vert x-y\Vert\geq \varepsilon$. Then $\frac{x+y}{2}\in (1-\delta)K$. Then
$$\frac{x^*(x)+x^*(y)}{2}\leq 1-\delta,$$
from where
$$x^*(y)\leq 1-2\delta.$$
Summarising we have proved that if $\Vert x-y\Vert\geq \varepsilon$ then $y\notin S:=\{z\in K: x^*(z)>1-2\delta\}$, which is a slice of $K$ since $\alpha>0$. This is equivalent to the following: if $x^*(y)\geq 1-2\delta=x^*(x)-2\delta$ then $\Vert x-y\Vert<\varepsilon$. Since $\varepsilon>0$ was arbitrary, we get that $x^*$ strongly exposes $x$ and the proof is complete.

\end{proof}

\begin{coro}\label{cor:carauconvexdenting}
Let $K$ be a uniformly convex subset of a Banach space $X$ and $\mathcal U$ be a CI ultrafilter on an infinite set $I$ be a Banach space. Then $\ext(K_\mathcal U)=\dent(K_\mathcal U)$ and $\exp(K_\mathcal U)=\strexp(K_\mathcal U)$. 
\end{coro}

\begin{proof}
By Proposition~\ref{unifconvex}, $K_\mathcal U$ is uniformly convex. The result follows from the previous proposition.
\end{proof}

\begin{coro}
Let $K$ be a uniformly convex subset of a Banach space $X$, $\mathcal U$ be a CI ultrafilter on an infinite set $I$ and $x\in K$. Then $x\in\ext(K)$ if and only if $\J(x)\in\ext(K_\mathcal U)$.
\end{coro}

\begin{proof}
It follows from Proposition~\ref{ext_unifconvex} and Theorem~\ref{th:jstrext}.
\end{proof}

The following diagram summarises the implications between the properties of $x\in K$ and the ones of $\J(x)\in K_{\mathcal U}$ for a uniformly convex set $K$. The situation is much simpler than in the general case:
\vspace{0.3cm}

\begin{center}
\adjustbox{scale=0.8,center}{%
\begin{tikzcd}[arrows=Rightarrow]
\text{$x\in\exp(K)$} \arrow[r,Leftrightarrow] \arrow[d,Leftrightarrow] & \text{$x\in\strexp(K)$} \arrow[r] \arrow[d,Leftrightarrow] & \text{$x\in\dent(K)$} \arrow[r,Leftrightarrow] \arrow[d,Leftrightarrow] & \text{$x\in\strext(K)$} \arrow[r,Leftrightarrow] \arrow[d,Leftrightarrow] & \text{$x\in\ext(K)$}  \\
\text{$\J(x)\in\exp_{(X^*)_\mathcal U}(K_\mathcal U)$} \arrow[r,Leftrightarrow] &
\text{$\J(x)\in\strexp_{(X^*)_\mathcal U}(K_\mathcal U)$} \arrow[r] & \text{$\J(x)\in\dent_{(X^*)_\mathcal U}(K_\mathcal U)$} \arrow[r,Leftrightarrow] & \text{$\J(x)\in\strext(K_\mathcal U)$} \arrow[r,Leftrightarrow] & \text{$\J(x)\in\ext(K_\mathcal U)$} \arrow[u,Leftrightarrow]\\
\end{tikzcd}}
\end{center}

\begin{rema}\label{remark:afinfuertementexpuesto}
In general we do not know whether every extreme point is a strongly exposed point in a uniformly convex set $K$. However, it turns out that every extreme point is strongly exposed in a sense which depends on the Minkowski functional $|\cdot|_K$. We say that a point is \textit{intrinsically strongly exposed} if there exists a linear functional $f\colon X\longrightarrow \mathbb R$ (not necessarily bounded) so that $f(x)=\sup_K f=1$ and that, for every $\varepsilon>0$, there exists $\delta>0$ so that $f(y)>1-\delta\Rightarrow \Vert x-y\Vert<\varepsilon$.

Let us prove that if $x\in K$ satisfies that $|x|_K=1$ then it is intrinsically strongly exposed. Define $Y=(\text{span}(K),|\cdot|_K)$. It is well known that $Y$ is a Banach space such that $B_Y=K$. By the Hahn-Banach theorem, there exists $f\in S_{Y^*}$ such that $f(x)=|x|_K=1$. Let us prove that $f$ strongly exposes $x$ in the above sense. To do so, pick $\varepsilon>0$ and take $\delta$ of the definition of uniformly convex set. Now if $y\in K$ satisfies that $f(y)>1-\delta$ we get
$$1-\frac{\delta}{2}<\frac{f(x+y)}{2}\leq \frac{|x+y|_K}{2}=\left|\frac{x+y}{2}\right|_K,$$
which means that $\frac{x+y}{2}\notin (1-\delta)K$. This implies that $\Vert x-y\Vert<\varepsilon$. The arbitrariness of $\varepsilon$ implies that every point of $\{x\in K: |x|_K=1\}$ is strongly exposed (actually, it is uniformly strongly exposed).

Note that $f$ is $|\cdot|_K$-continuous. However, in the general case, $f$ is not necessarily $\|\cdot\|$-continuous (unless $\text{span}(K)$ is closed since in that case the norm induced by $X$ on $Y$ and $|\cdot|_K$ are equivalent by the open mapping theorem).
\end{rema}

\begin{rema}
In spite of the previous remark, we can at least prove that $\strexp(K)$ is dense in $\ext(K)$ in a uniformly convex set $K$. Indeed, since $K$ is weakly compact, $K=\cconv(\strexp(K))$ and so $\dent(K)\subset \overline{\strexp(K)}$. By Proposition~\ref{ext_unifconvex}, $\ext(K)=\dent(K)$ provided that $K$ is uniformly convex.
\end{rema}

Note that, by applying Corollary \ref{cor:carauconvexdenting} and Theorem \ref{prop:caraextrefiltrofijo}, we can now give a characterisation of the denting points of a uniformly convex set.

\begin{theo}\label{theo:dentinguconvex}
Let $X$ be a Banach space, $K$ be a uniformly convex subset of $X$ and $\mathcal U$ be a free ultrafilter over $I$. Let $(x_i)_\mathcal U\in C_\mathcal U$. The following assertions are equivalent:
\begin{itemize}
    \item [(i)] $(x_i)_\mathcal U\in\dent(K_\mathcal U)$,
    \item [(ii)] for any $(y_i)_\mathcal U,(z_i)_\mathcal U\in C_\mathcal U$ so that $\lim_\mathcal U\Vert x_i-\frac{y_i+z_i}{2}\Vert=0$, it follows that $\lim_\mathcal U \Vert x_i-y_i\Vert=0$ and $\lim_\mathcal U\Vert x_i-z_i\Vert=0$.
\end{itemize}
\end{theo}

Recall that $F$ is a face of a convex set $C$ if for any $x,y\in C$ such that $(x,y)\cap F\neq\emptyset$, one has that $[x,y]\subset F$. A face $F$ of $C$ is proper if $F\neq\emptyset$ and $F\neq C$.

\begin{lema}\label{lemma:face}
Let $C$ be a symmetric bounded closed convex subset of a Banach space $X$. If $F$ is a proper face of $C$ then $F\subset |\cdot|_C^{-1}(\{1\})$. If moreover $C$ is strictly convex, then $F$ is a singleton. 
\end{lema}

\begin{proof} Note first that $0\notin F$, otherwise it follows easily that $C=F$, a contradiction. 
Now, suppose there exists $x\in F$ such that $|x|_C<1$. Thus $x$ belongs to the non-trivial segment $(0, x/|x|_C)$ in $C$. It follows that $0\in F$, so we obtain again that $C=F$, a contradiction.

Finally, assume that $C$ is strictly convex and take $x,y\in F$. We have that $|x|_C=|y|_C=1$ and, since $\frac{x+y}{2}\in F$, we also have that $\left|\frac{x+y}{2}\right|_C=1$. Then $x=y$ by strict convexity, so $F$ is a singleton.
\end{proof}

\begin{thm}
Let $\mathcal U$ be a free ultrafilter on an infinite set $I$. Let $K\subset X$ be a uniformly convex set such that $K_\mathcal U$ separates points of $(X^*)_\mathcal U$. Then 
\[(\exp(K))_\mathcal U\subset\exp_{(X^*)_\mathcal U}(K_\mathcal U)=\strexp_{(X^*)_\mathcal U}(K_\mathcal U).\]
\end{thm}

\begin{proof}
First, note that the equality follows from Proposition~\ref{ext_unifconvex}. Let $(x_i)_\mathcal U\in (\exp(K))_\mathcal U$. We can obviously suppose that $x_i\in \exp(K)$ for all $i\in I$. Let $x_i^*\in S_{X^*}$ that exposes $x_i$. Define a face of $K_\mathcal U$ by $$F=\left\{(y_i)_\mathcal U\in K_\mathcal U\ |\ \langle (x_i^*)_\mathcal U, (y_i)_\mathcal U\rangle=\sup_{K_\mathcal U}(x_i^*)_\mathcal U\right\}.$$ It is clear that $(x_i)_\mathcal U\in F$ and, in particular, $F\neq\emptyset$. Since $K_\mathcal U$ is strictly convex, in order to apply the previous lemma, we need to show that $F$ is a proper face. If $F=K$, then ${(x^*_i)_\mathcal U}_{|K_\mathcal U}=0$ since $0\in K$. By hypothesis, it follows that $(x^*_i)_\mathcal U=0$, which is a contradiction since $\|(x^*_i)_\mathcal U\|=1$. By Lemma~\ref{lemma:face}, we conclude that $F=\{(x_i)_\mathcal U\}$, i.e. $(x_i)_\mathcal U$ is exposed by $(x_i^*)_\mathcal U$.
\end{proof}

\textbf{Acknowledgements.} The authors are very grateful to the referees for their comments and particularly for the argument given by one of the referees leading to an improvement of Corollary \ref{improvement} in the case that the ultrafilter is CI (see Theorem \ref{th:jstrexp}). Moreover it allowed us to obtain Theorem \ref{theo:expequistronglyCI}.

This work was produced during the PhD thesis of the second author which is very grateful to his supervisor Matías Raja for his support and for initiating him to the theory of super weakly compact sets. The second author is also very grateful to the University of Zaragoza where this work was initiated.

The authors also want to thank Antonio Avil\'es and Matías Raja for fruitful conversations.

The research of L. Garc\'ia-Lirola was supported by the grants MTM2017-83262-C2-2-P and Fundaci\'on S\'eneca Regi\'on de Murcia 20906/PI/18.

The research of G. Grelier was supported by the Grants of Ministerio de Economía, Industria y Competitividad MTM2017-83262-C2-2-P; Fundación Séneca Región de Murcia 20906/PI/18; and by MICINN 2018 FPI fellowship with reference PRE2018-083703, associated to grant MTM2017-83262-C2-2-P. 

The research of A. Rueda Zoca was supported by Juan de la Cierva-Formaci\'on fellowship FJC2019-039973, by MTM2017-86182-P (Government of Spain, AEI/FEDER, EU), by MICINN (Spain) Grant PGC2018-093794-B-I00 (MCIU, AEI, FEDER, UE), by Fundaci\'on S\'eneca, ACyT Regi\'on de Murcia grant 20797/PI/18, by Junta de Andaluc\'ia Grant A-FQM-484-UGR18 and by Junta de Andaluc\'ia Grant FQM-0185.

\end{document}